\documentclass{amsart}
\usepackage{amssymb,amsmath,amscd,amsthm}
\usepackage{amsfonts,epsfig,latexsym,graphicx,amssymb}
\usepackage{color}
\usepackage{enumerate}
\usepackage{todonotes}

\usepackage{hyperref}

\date{\today}

\newcommand{\R}{{\mathbb R}}

\newcommand{\N}{{\mathbb N}}

\newcommand{\mK}{\mathcal{K}}  % The compact of measures

\newcommand{\mN}{\mathcal{N}}

\newcommand{\Man}{\mathcal{M}}

\newcommand{\st}{\widetilde{S}}
\newcommand{\pd}{s}

\newcommand{\Lip}{{\mathrm{Lip}}}

\newcommand{\SL}{{\mathrm{SL}}}
\newcommand{\Rot}{{\mathrm{Rot}}}
\newcommand{\dd}{\, \mathrm{d}}
\newcommand{\pp}{[\cdot]}

\newcommand{\rep}{\mathbf{r}}
\newcommand{\CM}{C_{\Man}}
\newcommand{\mKM}{\mK_{\Man}}

\newcommand{\txi}{\widetilde{\xi}}
\newcommand{\tR}{\widetilde{R}}

\newcommand{\bC}{\overline{C}}

\newcommand{\RP}{\mathbb{RP}}

\newcommand{\avg}{a}

\newcommand{\Ind}{\mathbf{1}}

\newcommand{\Cxi}{C_{2}}
\newcommand{\CVar}{C_{\sigma}}

\newcounter{mit}

\newcounter{Bit}

\newtheorem{theorem}{Theorem}[section]
\newtheorem{lemma}[theorem]{Lemma}
\newtheorem{prop}[theorem]{Proposition}
\newtheorem{coro}[theorem]{Corollary}

\theoremstyle{definition}
\newtheorem{remark}[theorem]{Remark}

\newtheorem{example}[theorem]{Example}
\newtheorem{defi}[theorem]{Definition}

\newcommand{\emr}[1]{\textcolor{black}{#1}}

\newcommand{\mgr}{\mu}
\newcommand{\msp}{\nu}

\newcommand{\rad}{\rho}

\def\dist{{\rm dist}}
\def\supp{\mathop{\rm supp}}

\newcommand{\bi}{{\bf i}}

\def\N{{\mathbb N}}

\newcommand{\E}{{\mathbb E}\,}
\newcommand{\Prob}{{\mathbb P}\,}
\newcommand{\var}{\mathop{\rm Var}\nolimits}
\def\P{\Prob}

\def\Homeo{{\text{Homeo}}}

\def\tD{\tilde D}
\def\tB{\tilde B}
\def\bK{{\bf K}}

\newcommand{\eps}{{\varepsilon}}

\newcommand{\const}{{\rm const}}

\begin{document}

\title[Central Limit Theorem for non-stationary random matrix products]{Central Limit Theorem for non-stationary random products of $\SL(2, \R)$ matrices}

%\title[Non-stationary version of Furstenberg Theorem]{Non-stationary version of Furstenberg Theorem\\ on Random Products of $\SL(m, \mathbb{R})$ matrices}

%

\author[A.\ Gorodetski]{Anton Gorodetski}

\address{Department of Mathematics, University of California, Irvine, CA~92697, USA}

\email{asgor@uci.edu}

\thanks{A.\ G.\ was supported in part by NSF grant DMS--2247966.}

\author[V. Kleptsyn]{Victor Kleptsyn}

\address{CNRS, Institute of Mathematical Research of Rennes, IRMAR, UMR 6625 du CNRS}

\email{victor.kleptsyn@univ-rennes1.fr}

\thanks{V.\ K.\ was supported in part by ANR Gromeov (ANR-19-CE40-0007) and by Centre Henri Lebesgue (ANR-11-LABX-0020-01)}

\author[G. Monakov]{Grigorii Monakov}

\address{Department of Mathematics, University of California, Irvine, CA~92697, USA}

\email{gmonakov@uci.edu}

\begin{abstract}
We prove Central Limit Theorem for non-stationary random products of $\SL(2, \mathbb{R})$ matrices, generalizing the classical results by Le Page and Tutubalin that were obtained in the case of iid random matrix products.
\end{abstract}

%\begin{keyword}[class=MSC]
%\kwd[Primary ]{60F05}
%\kwd{60B20}
%\kwd[; secondary ]{37H15}
%\kwd{37A50}
%\end{keyword}
%
%\begin{keyword}
%\kwd{Central Limit Theorem}
%\kwd{random matrix products}
%\end{keyword}
%
%\end{frontmatter}

\maketitle

%\tableofcontents

%\section*{NOTATIONS}
%\begin{itemize}
%\item $a$ - parameter;
%\item $J$ - interval of parameters;
%\item $\mu$ - measure in the topological space $\Omega$;
%\item $\mu_a$ - measure in the space of matrices, $\mu_a=F_a(\mu)$;
%\item $\mathbb{P}$ - measure $\mu^{\mathbb{N}}$ in the space of sequences $\Omega^\mathbb{N}$;
%\item $\nu_a$ - stationary measure on $S^1\cong \mathbb{RP}^1$ that corresponds to the parameter $a\in J$;
%\item $\lambda_F(a)$ - "Furstenberg" value of the Lapunov exponent;
%\item $T_{n,a, \bar \omega}=F_a(\omega_1)F_a(\omega_2)\ldots F_a(\omega_n)$
%\item $f_{n, a, \bar \omega}:\mathbb{RP}^1\to \mathbb{RP}^1$, $\mathbb{RP}^1\cong S^1$, is the map induced by $T_{n,a, \bar \omega}:\mathbb{R}^2\to \mathbb{R}^2$
%\item $\tilde f_{n, a, \bar \omega}:\mathbb{R}^1\to \mathbb{R}^1$ is the lift of $f_{n, a, \bar \omega}:S^1\to S^1$
%\item $\Lambda_n$ - grid in the space of parameters;
%\item $\gamma$ - small constant needed to consider the interval $[\gamma n, n]$ of numbers of iterates;
%\item $\delta$ and $\varepsilon$ - constants from Lemma \ref{l.est};
%\item $\theta$ - determines the size of the grid $\Lambda_n$, $|\Lambda_n|=e^{\theta n}$.
%\end{itemize}
%
%\newpage

%\tableofcontents

\section{Introduction}

\subsection{Historical background}

The two most fundamental results in probability that are present in almost every textbook are the (strong) Law of Large Numbers (LLN) and the Central Limit Theorem (CLT). In the most basic form, if $\{\xi_n\}$ is an iid sequence of random variables with finite expectation $\avg$ and finite variance $\sigma^2$, the LLN claims that almost surely $\frac{1}{n}\sum_{i=1}^n \xi_i\to \avg$, and the CLT claims that $\frac{\sum_{i=1}^n \xi_i-n\avg }{\sqrt{n}\sigma}$ converges in distribution to a normal distribution $\mathcal{N}(0,1)$ with mean~$0$ and variance~$1$.

There are many ways to relax the assumptions in both cases. In particular, the random variables do not have to be identically distributed. For example, a non-stationary version of the LLN known as {\it Kolmogorov's Law} \cite{Kol} claims that
if $\{\xi_i\}$ is a sequence of independent random variables with $\avg_i=\mathbb{E}\xi_i$, $\sigma_i^2=\text{\rm Var}\,(\xi_i)$, and $\sum_{i=1}^\infty \frac{\sigma_i^2}{i^2}<\infty$, then $\frac{\sum_{i=1}^n(\xi_i-\avg_i)}{n}\to 0$ almost surely. On the other hand, if for some $\delta>0$ the sequence $\mathbb{E}|\xi_i|^{2+\delta}$ is uniformly bounded, then the sequence of random variables $\frac{\sum_{i=1}^n(\xi_i-\avg_i)}{\sum_{i=1}^n\sigma^2_i}$
converges in distribution to $\mathcal{N}(0,1)$.

There are plenty of different generalizations and forms of these statements. For example, for some of the analogs of the LLN and CLT for the sums of iid random variables in the context of random walks on groups see the survey \cite{F} and monograph \cite{BQ2}, and references therein. Here we discuss random matrix products. In this case, a multiplicative version  of the LLN is given by Furstenberg and Kesten \cite{FK}. A stronger result is the famous Furstenberg Theorem, which also guarantees positivity of the Lyapunov exponent:

\begin{theorem}[H. Furstenberg \cite{Fur1}]\label{t.F}
Let $\{X_k, k\ge 1\}$ be independent and identically distributed random variables, taking values in $\SL(d, \mathbb{R})$, the $d\times d$ matrices with determinant one, let $G_X$ be the smallest closed subgroup of $\SL(d, \mathbb{R})$ containing the support of the distribution of $X_1$, and assume that
$$
\mathbb{E}[\log\|X_1\|]<\infty.
$$
Also, assume that $G_X$ is not compact and is strongly irreducible, i.e. there exists no $G_X$-invariant finite union of proper subspaces of~$\R^d$. Then there exists a positive constant $\lambda_F$ (Lyapunov exponent) such that with probability one
$$
\lim_{n\to \infty}\frac{1}{n}\log\|X_n\ldots X_2X_1\|=\lambda_F>0.
$$
\end{theorem}
\begin{remark}
In the case of random products of $\SL(2, \mathbb{R})$ matrices, the assumption that~$G_X$ is not compact and is strongly irreducible is equivalent to the assumption that there exists no measure on $\mathbb{RP}^1$ invariant under the action of every map from $G_X$, see \cite[Lemma~3.6]{AB}.
\end{remark}

The CLT for the products of iid random matrices is also available. The initial results were obtained for matrices with positive coefficients \cite{Bel}, \cite{FK}.  In the case of absolutely continuous distributions it was obtained by Tutubalin \cite{T1, T2}. The requirements on regularity of distributions was relaxed by Le Page, who proved the CLT for random matrix products under the assumption of finite exponential moments \cite{L}, see also \cite{BL}, \cite{GG}, \cite{GR}, \cite{GM},~\cite{J}. Finally, the assumption on the moments of the distribution  was optimized by Benoist and Quint \cite{BQ1}:

\begin{theorem}[Benuist, Quint, \cite{BQ1}]\label{t.BQ}
    Let $\{X_k, k \ge 1\}$ be independent and identically distributed random matrices in $\SL(d, \R)$. %, distributed with respect to measure $\mgr$.
    Assume that $G_X$ is non-compact and  %\textit{proximal}\footnote{Let us recall that a group action is proximal if for any pair of points (in our case, in $\mathbb{RP}^{d-1}$), there exists a sequence of group elements that maps those points arbitrarily close together.} and
    \textit{strongly irreducible}
    and
    \begin{equation} \label{LogMom}
        \mathbb{E}\left[(\log \|X_1\|)^2 \right]
        %= \int_{\SL(d, \R)} (\log \|X\|)^2 \dd \mgr(X)
        < \infty.
    \end{equation}
    Then there exists $\sigma>0$ such that the random variables
    \begin{equation*}
        \frac{\log \|X_n\ldots X_1\|-n\lambda_F}{\sqrt{n}},
    \end{equation*}
    where $\lambda_F > 0$ is the Lyapunov exponent, converge in distribution to $\mathcal{N}(0,\sigma^2)$.
\end{theorem}

Notice that both Theorems \ref{t.F} and \ref{t.BQ} require the sequence of random matrices to be identically distributed.  That requirement allows to consider a stationary measure for the random dynamics on the projective space, which is a key notion used in the proofs of both results.  Nevertheless, the classical LLN and CLT for sums of real valued random variables hold without that assumption, and it is natural to expect that non-stationary versions of the LLN and CLT for random matrix products  should hold as well. Indeed, the non-stationary version of the Furstenberg Theorem was recently provided in~\cite{GK}, and it already found interesting applications in spectral theory~\cite{GK2}. The non-stationary version of the CLT for random products of $\SL(2, \mathbb{R})$ matrices is the main result of this paper.

\subsection{Preliminaries and main results}\label{s:prelim}

Let us now provide the setting needed to state our main result. From now on, let us restrict ourselves to the case of products of $\SL(2, \mathbb{R})$ matrices.

Let $\mK$ be a compact subset in the set of probability measures on the group~$\SL(2,\R)$.
%For any $A\in \SL(2, \mathbb{R})$ we will denote by $f_A:\mathbb{RP}^{1}\to \mathbb{RP}^{1}$ the induced projective transformation.
We will say that {\bf the measures condition} is satisfied if for every measure $\mgr\in \mK$ there are no Borel probability measures $\msp_1$, $\msp_2$ on $\mathbb{RP}^{1}$ such that $(f_A)_*\msp_1=\msp_2$ for $\mgr$-almost every $A\in \SL(2, \mathbb{R})$,
\emr{where $f_A$ is the projectivisation of the matrix~$A$ (see Eq.~\eqref{f-B-def} below).}

Let us fix some sequence $\{\mgr_i\}_{i\in \mathbb{N}}$, $\mgr_i\in \mK$, and let $A_i\in \SL(2, \mathbb{R})$ be independent matrix-valued random variables, with $A_i$ being distributed w.r.t. $\mgr_{i}$. Set
\[
T_n=A_nA_{n-1}\ldots A_1,
\]
and denote
\begin{equation}\label{eq:L-def}
L_n=\mathbb{E}\log\|T_n\|.
\end{equation}
%where the expectation is taken over the distribution $\nu_{1}\times \nu_{2}\times\ldots\times \nu_{n}$.

If the measures condition is satisfied, then for any $\{\mgr_i\}_{i\in \mathbb{N}}$, $\mgr_i\in \mK$, the sequence $\{L_n\}$ must grow at least linearly, i.e. the norms of the random products must grow exponentially on average, see \cite[Theorem 1.5]{GK}. A related statement on exponential growth of the norms in the case of non-stationary linear cocycles over Markov chains was established by Goldsheid \cite{G}. Moreover, if additionally a uniform bound on some exponential moment exists  for distributions from $\mathcal{K}$,  the non-random sequence $\{L_n\}$ describes the behavior of almost every random product, and in this sense serves as a non-stationary analog of Lyapunov exponent. Namely, almost surely one has $\lim_{n\to \infty}\frac{1}{n}\left(\log\|T_n\|-L_n\right)=0,$ see  \cite[Theorem 1.1]{GK}.  This provides a direct analog of the LLN for non-stationary random matrix products.

That compels the question whether an analog of CLT for non-identically distributed random variables must hold in this setting.
Our main result provides a positive answer in dimension two: %is exactly the CLT for random non-stationary products of $\SL(2, \mathbb{R})$ matrices:

\begin{theorem}\label{t.CLT}
 Let $\mK$ be a compact subset in the set of probability measures on the group~$\SL(2,\R)$ that satisfies the measures condition, and there exists $\gamma \emr{\, > 2}$ and $M>0$ such that for any $\mgr \in \mK$ one has
\begin{equation} \label{tails}
    \E_{\mgr} (\log \|A\| )^{\gamma} < M.
\end{equation}
Then the random variables
    $$
        \frac{\log \|T_n\|-L_n}{\sqrt{\var(\log\|T_n\|)}}
    $$
    converge in distribution to $\mathcal{N}(0,1)$, with the convergence that is uniform with respect to the choice of the sequence $\mu_1,\mu_2,\dots\in\mK$.

     Also, there are constants $C_1, \Cxi>0$ and an index $n_0$ such that for all $n\ge n_0$ and all $\mu_1,\dots,\mu_n\in\mK$ one has
    \begin{equation} \label{VarGrowth}
        C_1^2{n}\le \var(\log\|T_n\|)\le \Cxi^2{n}.
    \end{equation}

\end{theorem}

\begin{remark}\label{rk:CLT}
    \begin{enumerate}[(a)]
      \item  The condition~(\ref{tails}) with the assumption $\gamma>2$ is optimal, in a sense that it cannot be strengthened to $\gamma=2$. This is in contrast with the iid case, compare with Theorem~\ref{t.BQ}. We discuss this below, see Example~\ref{ex:no-CLT-2D} in Section~\ref{s:examples}.
      \item One should expect that, under suitable conditions, Theorem \ref{t.CLT} should hold for random $\SL(d, \mathbb{R})$ matrix products for every $d\ge 2$. To prove such a statement, it would be helpful to have a non-stationary analog of simplicity of the Lyapunov spectrum, see \cite{GR}, \cite{GM} for the case of iid random matrix products. In the case of some specific regular distributions in $\SL(d, \mathbb{R})$ such an analog was recently established \cite{AFGQ}, but a statement for a general sequence of distributions is currently not available, even if certainly expected.
%      \item\label{i:CLT-bis}
         \end{enumerate}
\end{remark}
The conclusion of Theorem~\ref{t.CLT} also applies to the distribution of log-lengths of individual vectors, $\log |T_n v|$, as well as to the matrix elements $(T_n)_{i,j}$.

\begin{theorem}\label{t:CLT-vectors}
Under the assumptions of Theorem~\ref{t.CLT}, for any nonzero $v\in \R^2$ and any $i,j=1,2$ the random variables
    $$
        \frac{\log |T_n v|-L_n}{\sqrt{\var(\log\|T_n\|)}}, \quad         \frac{\log |(T_n)_{ij}| - L_n}{\sqrt{\var(\log\|T_n\|)}}
    $$
    converge in distribution to $\mathcal{N}(0,1)$, where $(T_n)_{ij}$ is the $(i,j)$-th element of the matrix~$T_n$.
\end{theorem}

 We consider this paper to ba a ``proof of concept'', a demonstration that an enormous amount of results on random walks on groups formulated in terms of the law of large numbers, CLT, the law of the iterated logarithms etc. can be expected to hold in the non-stationary setting, even if the notion of the stationary measure on the projective space is not defined. The key observation here is that a random dynamical system acts on the measures on the phase space by convolutions, i.e. averaging of the push-forwards of the measure by the random dynamics, and such an action ``moves'' measures toward the space of measures with some specific modulus of continuity, e.g. H\"older or log-H\"older, depending on the setting, see \cite[Theorem 2.8]{GKM}, \cite[Theorems 2.4 and 2.9]{M}. For some other recent results related to non-stationary random dynamics see \cite{GK3}, \cite{M2}, \cite{M3}.

%\newpage

\subsection{Notations and plan of the proof of the main result}\label{s.proofCLT}

Let us introduce some notations. Let
\[
T_{(n_1,n_2]}:=A_{n_2}A_{n_2 -1}\dots A_{n_1+1}
\]
be the part of the product of our random matrices $A_i$, where  the index varies from $n_1+1$ to $n_2$. Also, denote
$$
\xi_n= \log\|T_n\|, \quad \xi_{(n_1, n_2]}=\log \left\| T_{(n_1, n_2]} \right\|.
$$

Note that if two intervals of indices $(n_1,n_2]$ and $(n_1',n_2']$ are disjoint, then the corresponding products $T_{(n_1,n_2]}$ and $T_{(n_1',n_2']}$ are independent, and thus so are their log-norms~$\xi_{(n_1,n_2]}$ and $\xi_{(n_1',n_2']}$.

Now, a long product of matrices can be split into two parts (that we will later choose to be of comparable lengths): for any $n,n'$ one has
\[
T_{n+n'}=T_{(n,n+n']}T_{n} ;
\]
in particular, this implies
\begin{equation}\label{eq:sub-add}
\xi_{n+n'}= \log \| T_{(n,n+n']} T_{n} \| \le \log \| T_{n}\| + \log \| T_{(n,n+n']}\| = \xi_n + \xi_{(n,n+n']}.
\end{equation}

The right hand side of the inequality in~\eqref{eq:sub-add} is a sum of two independent random variables; let us introduce the random variable $R_{n,n'}$ that measures the difference between the right and left hand sides of~\eqref{eq:sub-add}:
\begin{equation}\label{e.Rnn}
    R_{n, n'} = \log \|T_n\| + \log \left\| T_{(n, n + n']} \right\| - \log \| T_{n + n'} \|= (\xi_n + \xi_{(n,n+n']}) - \xi_{n+n'}.
\end{equation}

%the bounds~\eqref{VarGrowth}. To do it, the first step is

We start the proof of Theorem~\ref{t.CLT} with establishing uniform moment bounds for the discrepancy~$R_{n,n'}$; this is done in Sec.~\ref{s:R-estimates}, see Proposition~\ref{prop:Rmoments}. To do so, we have to show that it is (sufficiently) improbable that the most expanded vector for the product $T_n$ is sent to the direction close to the one that is contracted by $T_{(n,n+n']}$. This can be reformulated in terms of the action on the projective line~$\mathbb{RP}^1$: in these terms, it is the probability \emr{of two sequences of iterations sending two given points close to each other}. We use results from~\cite{M}, where such estimates (log-H\"older bounds after a finite number of non-stationary iterations) were established.%, to obtain Lemma~\ref{lm:Rn}, providing such tail estimates.

Next, we use these estimates to establish a control on the central moments of~$\xi_n$, using the relation
\begin{equation}\label{eq:xi-sum}
\xi_{n+n'}=
\left( \xi_n + \xi_{(n,n+n']} \right) - R_{n,n'}.
\end{equation}
To do so, we use the fact that the sum in the parenthesis is a sum of independent random variables, and the moments for $R_{n,n'}$ are uniformly bounded, thus its addition cannot increase the moments too much. This is done in Section~\ref{s:xi-moments}, see Proposition~\ref{prop:ximoments}. Also (see Lemma~\ref{l:var-lower}), we get a lower bound for the linear growth of the variances $\var \xi_n$, thus altogether establishing the conclusion~\eqref{VarGrowth}. The argument is again based on using~\eqref{eq:xi-sum}; a key difficulty here is to establish the arbitrarily large lower bound for the variances. The latter is Proposition~\ref{p:large}, whose proof (that turned out to be surprisingly technical) is provided in Section~\ref{section:variance}.
%~\ref{s:xi-moments-upper}
%Then, in Section~\ref{s:linear-var}

The final step in the proof of Theorem~\ref{t.CLT} is a bootstrapping argument, provided in Sections~\ref{s.Distance} and~\ref{s:main}. Namely, the sum of two independent random variables (for instance, $\xi_n$ and $\xi_{(n,n+n'])}$) is closer to the Gaussian behavior than the summands separately. We introduce the quantitative way (\ref{e.Nr}) of measuring how close the distribution is to the Gaussian, and establish the corresponding inequality in Section~\ref{s:sum-independent}. Then, we control how an additional perturbation, coming from the $R_{n,n'}$ term, can worsen the bounds. This is done in Section~\ref{s:correction}.

We conclude by joining the bootstrapping estimates with the bounds established for $\xi_n$ and $R_{n,n'}$, and complete the proof of Theorem~\ref{t.CLT} in Section~\ref{s:main}.

We will deduce Theorem~\ref{t:CLT-vectors} from Theorem~\ref{t.CLT} in Section~\ref{s:CLT-vectors} below.

\section{Moment estimates for $R_{n,n'}$}\label{s:R-estimates}

In this section we provide the estimates on the moments of discrepancies~$R_{n,n'}$ defined by (\ref{e.Rnn}). We start by discussing some properties of $\SL(2, \R)$ matrices in Section~\ref{s:linear}, and then state and prove the main estimate, Proposition~\ref{prop:Rmoments}, in Section~\ref{s:moments}.

\subsection{Preliminaries: action of $\SL(2,\R)$ matrices}\label{s:linear}

Let $\pp$ be the canonical projection
 $\pp: \R^2\backslash \{0\} \to \RP^1$, and denote by $f_B$ the projectivization of the matrix $B$, namely,
\begin{equation}\label{f-B-def}
    f_B: \RP^1 \to \RP ^1, \quad \text{such that }  f_B \circ \pp = \pp \circ B.
\end{equation}
We consider $\RP^1$ to be equipped with the metric $\dist(\cdot,\cdot)$ that is the angle between the corresponding lines.

     Recall that every matrix $B\in \SL(2,\R)$ can be written as a product
    \begin{equation}\label{eq:B-diag}
        B= \Rot_{\beta_1} \left( \begin{smallmatrix}
            \|B\| & 0 \\
            0 & \|B\|^{-1}
        \end{smallmatrix} \right) \Rot_{\beta_2},
    \end{equation}
    where $\Rot_{\beta}$ is a rotation by the angle $\beta$. Let $e_1,e_2$ be the standard basis vectors of~$\R^2$; for a matrix~$B$ in the form~\eqref{eq:B-diag}, denote by $\rep(B)=[\Rot_{\beta_2}^{-1} e_2]\in\RP^1$ the direction that is \emph{contracted} the most by $B$. Finally, for
    a vector $v \in \R^2$ and a matrix $B \in \SL(2, \R)$, let
\begin{equation*}
    \Theta(B, v) = \log(\|B\| \cdot |v|) - \log(|Bv|)= \log \frac{\|B\| \cdot |v|}{|Bv|};
\end{equation*}
in other words, $\Theta(B, v)$ is a function that compares the expansion by $B$ of the vector $v$ with the maximal possible expansion by $B$ over all the nonzero vectors.

We have the following estimate:
\begin{lemma}\label{l:Theta-bound}
For any $B\in \SL(2,\R)$ and any nonzero vector $v\in \R^2$, one has
\begin{equation}\label{eq:Theta-upper}
\Theta(B, v)\le -\log \sin\dist([v],\rep(B)).
\end{equation}
\end{lemma}
\begin{proof}
We can assume the vector $v$ to be of unit length. Distance $\dist([v],\rep(B))$ is equal to the angle between $\Rot_{\beta_2}(v)$ and~$e_2$, thus the component of $\Rot_{\beta_2}(v)$ that is parallel to $e_1$ is equal to $\sin\dist([v],\rep(B))$; see Fig.~\ref{fig:calculating-B}. After the application of the diagonal matrix in the representation~$\eqref{eq:B-diag}$, this component gets multiplied by $\|B\|$, and provides a lower bound for the length of the image~$|Bv|$. This immediately implies~\eqref{eq:Theta-upper}.
\end{proof}

\begin{figure}
\includegraphics{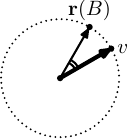} \qquad
\includegraphics{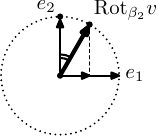} \qquad
\includegraphics{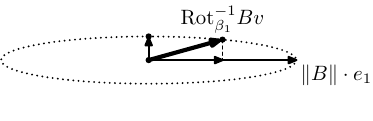}
\caption{Left: mostly contracted direction $\rep(B)$ and a vector~$v$. Center: their images after rotation by $\Rot_{\beta_2}$. Right:  the images after the application of the diagonal matrix in the decomposition~\eqref{eq:B-diag}.}\label{fig:calculating-B}
\end{figure}

Lemma~\ref{l:Theta-bound} immediately implies the following estimate:
\begin{coro}\label{cor:log-2}
For any $B\in \SL(2,\R)$ for at least one of two coordinate vectors~$e_i$ one has $\Theta(B,e_i)\le \log \sqrt{2}$.
\end{coro}
\begin{proof}
Indeed, at least one of the points $[e_1],[e_2]$ is at the distance at least $\frac{\pi}{4}$ from the direction $\rep(B)$, and the estimate follows from~\eqref{eq:Theta-upper}.
\end{proof}

Lemma~\ref{l:Theta-bound} involves the location of the direction~$\rep(B)$; the following statement gives a way to \emr{approximate its location:}
\begin{lemma}\label{l:B-preimage}
For any $B\in \SL(2,\R)$ for at least one of two coordinate vectors $e_j$ one has
\begin{equation}\label{eq:B-r-preimage}
\dist(f_{B}^{-1}[e_j],\rep(B)) \le \frac{1}{\|B\|^2}.
\end{equation}
\end{lemma}
\begin{proof}
One of the vectors $\Rot_{\beta_1}^{-1}e_j$ is at the angle $\theta\ge \frac{\pi}{4}$ with the direction $e_1$. Now,
\begin{equation}\label{eq:B-preimage}
\left( \begin{smallmatrix}
            \|B\| & 0 \\
            0 & \|B\|^{-1}
        \end{smallmatrix} \right)^{-1}
        \Rot_{\beta_1}^{-1} e_j =
\left( \begin{smallmatrix}
            \|B\| & 0 \\
            0 & \|B\|^{-1}
        \end{smallmatrix} \right)^{-1} \left( \begin{smallmatrix} \cos \theta \\ \pm \sin \theta \end{smallmatrix}\right) =
        \left( \begin{smallmatrix} \frac{1}{\|B\|} \cos \theta \\ \pm \|B\| \cdot \sin \theta \end{smallmatrix}\right).
\end{equation}
	The tangent of the angle $\theta'$ between the vector~\eqref{eq:B-preimage} and the direction of~$e_2$ is thus equal to $\tan \theta'= \frac{\mathop{\mathrm{cotan}} \theta}{\|B\|^2} \le \frac{1}{\|B\|^2}$. An application of $\Rot_{\beta_2}^{-1}$ then concludes the proof:
	\[
		\dist(f_{B}^{-1}[e_j],\rep(B)) = \theta' \le \tan \theta' \le \frac{1}{\|B\|^2}.
	\]
\end{proof}

Next, note that one can estimate the decrease in the log-norm in the product of two matrices~$B_1$ and~$B_2$ using their action on any (in particular, well-chosen) vector $v$:
\begin{lemma}\label{l:B1-B2}
For any $B_1,B_2\in \SL(2,\R)$ and any nonzero vector $v$,
\begin{equation}\label{eq:B1-B2}
\log \|B_1\| + \log \|B_2\| - \log \|B_2B_1\| \le \Theta(B_1,v) + \Theta(B_2, B_1 v).
\end{equation}
\end{lemma}
\begin{proof}
We can assume the vector $v$ to be a unit one. Then,
\begin{multline*}
\log \|B_2 B_1\| \ge \log |B_2 B_1 v| = \log \|B_2\| + \log |B_1 v| - \Theta(B_2, B_1 v)
\\
= \log \|B_2\| + \log \|B_1\| - \Theta(B_1 , v) -  \Theta(B_2, B_1 v).
\end{multline*}
\end{proof}

Finally, the previous lemmas can be joined together in order to obtain a good estimate for the right hand side of~\eqref{eq:B1-B2}. Namely, for any $B_1,B_2\in \SL(2,\R)$ consider the images of the coordinate directions $f_{B_1}([e_i])$, $i=1,2$, and preimages $f^{-1}_{B_2}([e_j])$, $j=1,2$. Let $\Delta_{B_1,B_2}$ be the minimal distance between these pairs:
\begin{equation}\label{eq:Delta}
\Delta_{B_1,B_2}:= \min_{i=1,2} \min_{j=1,2} \dist (f_{B_1}([e_i]),f^{-1}_{B_2}([e_j])).
\end{equation}

We then have the following estimate:
\begin{prop}\label{p:B1-B2}
For any $B_1,B_2\in \SL(2,\R)$, one has
\begin{equation}\label{eq:B1-B2-Delta}
\log \|B_1\| + \log \|B_2\| - \log \|B_1B_2\| \le \log 4\sqrt{2} + \log  \min(\Delta_{B_1,B_2}^{-1} , \|B_2\|^2).
\end{equation}
\end{prop}
\begin{proof}
Due to Corollary~\ref{cor:log-2}, we can choose a coordinate vector $e_i$ so that $\Theta(B_1,e_i)\le \log\sqrt{2}$. Due to Lemma~\ref{l:B1-B2}, the left hand side of~\eqref{eq:B1-B2-Delta} does not exceed
\begin{equation}\label{eq:B1-B2-Theta}
\Theta(B_1,e_i)+\Theta(B_2,B_1 e_i) \le \log \sqrt{2} + \Theta(B_2,B_1 e_i).
\end{equation}
Note that for any nonzero vector~$v$ (in particular, for $v=B_1e_i$), one has $|B_2 v|\ge \frac{1}{\|B_2\|} |v|$ and hence
\[
\Theta(B_2,v) \le \log\|B_2\|^2.
\]
Thus, if $\Delta_{B_1,B_2}^{-1}\ge \frac{1}{2}\|B_2\|^2$, the estimate~\eqref{eq:B1-B2-Theta} implies the desired~\eqref{eq:B1-B2-Delta} immediately: its right hand side then does not exceed
\[
\log \sqrt{2} + \Theta(B_2,B_1 e_i) \le \log \sqrt{2} + \log (2 \cdot \frac{1}{2}\|B_2\|^2) \le \log 2\sqrt{2} + \log  \min(\Delta_{B_1,B_2}^{-1} , \|B_2\|^2).
\]

Given that, from now on we can assume that $\Delta^{-1}_{B_1,B_2}< \frac{1}{2}\|B_2\|^2$, and the proof will be complete once we establish that
\begin{equation}\label{eq:Theta-Delta}
\Theta(B_2,B_1 e_i)\le \log 4 + \log \Delta_{B_1,B_2}^{-1}.
\end{equation}

\begin{figure}[h]
\includegraphics{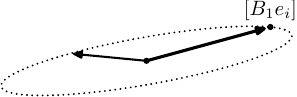} \qquad
\includegraphics{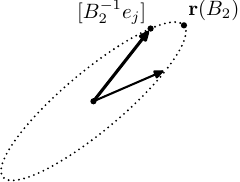} \\[5mm]
\includegraphics{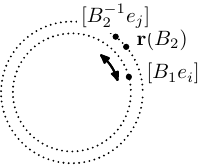}
\caption{Top left: image $[B_1 e_i]$ provided by Corollary~\ref{cor:log-2}. Top right: the direction $\mathbf{r}(B_2)$ and the preimage $[B_2^{-1} e_j]$, sufficiently close to it, that is provided by Lemma~\ref{l:B-preimage}. Bottom: these three directions and an arc of length at least $\Delta_{B_1,B_2}$.}\label{fig:vectors}
\end{figure}

Now, let $e_j$ be the coordinate vector for which the estimate~\eqref{eq:B-r-preimage} from Lemma~\ref{l:B-preimage} holds (see Fig.~\ref{fig:vectors}). Then for $v=f_{B_1}(e_i)$, by triangle inequality and definition of $\Delta_{B_1,B_2}$ one has
\begin{multline*}
\dist ([v], \rep(B_2))\ge \dist (f_{B_1}([e_i]), f_{B_2}^{-1}([e_j])) -  \dist (f_{B_2}^{-1}([e_j]),\rep(B_2))
\\
\ge \Delta_{B_1,B_2} - \frac{1}{\|B_2\|^2} \ge \frac{1}{2} \Delta_{B_1,B_2},
\end{multline*}
where we have used $\Delta_{B_1,B_2}> \frac{2}{\|B_2\|^2}$ for the last inequality.
Applying Lemma~\ref{l:Theta-bound}, we finally get
\[
\Theta(B_2, B_1e_i) \le -\log \sin \dist ([v], \rep(B_2)) \le -\log \sin \frac{1}{2}\Delta_{B_1,B_2} \le \log 4+ \log \Delta_{B_1,B_2}^{-1},
\]
where we have used the inequality $\sin x \ge \frac{2}{\pi} x \ge \frac{x}{2}$ for $x\in [0,\frac{\pi}{2}]$. This completes the proof of~\eqref{eq:Theta-Delta} and thus of the proposition.
\end{proof}

%Consider the projection $[v]$, and denote by $\ball_r([v])$ its radius $r$ neighbourhood in $\RP^1$. Finally, consider the projections $[e_1],[e_2]$ of the basis vectors $e_1, e_2 \in \R^2$.

\subsection{Estimates}\label{s:moments}
The main statement of this section is the following estimate for the moments of the random variables $R_{n,n'}$. \emr{As $\gamma>2$, we can choose $\eps\in (0,1]$ such that
\[
2+\eps<\gamma.
\]
We fix a choice of such $\eps$ from now on until the end of the paper.}
\begin{prop} \label{prop:Rmoments}
    Under the assumptions of Theorem \ref{t.CLT}, %{\bf standing assumption}
    there exists $C_R$, such that for every $n, n' \in \N$ with $\emr{\frac{n'}{2}\le }n \le 2n'$, \emr{and any measures $\mgr_1,\dots,\mgr_{n+n'}\in \mK$,} one has
    \begin{equation*}
        \E R_{n, n'} < C_R, \quad \E R_{n, n'}^2 < C_R^{\emr{2}}, \quad \text{and} \quad \E R_{n, n'}^{\emr{2+\eps}} < C_R^{\emr{2+\eps}}.
    \end{equation*}
\end{prop}

% The following lemma allows to control the ``loss of expansion'' $\Theta(B,v)$ in terms of closeness of the direction of $v$ to preimages of given directions:
%
%\begin{lemma} \label{lm:geom}
%    For $x > 0$ take $r = 2 e^{-x/2}$. If
%    \begin{equation*}
%        f_{B^{-1}}([e_1]) \notin \ball_r([v]) \quad \text{and} \quad f_{B^{-1}}([e_2]) \notin \ball_r([v])
%    \end{equation*}
%    then
%    \begin{equation} \label{thetaEst}
%        \Theta(B, v) \le x.
%    \end{equation}
%\end{lemma}

\begin{proof}[Proof of Proposition \ref{prop:Rmoments}]
    First of all, notice that it suffices to prove the estimate for~$\E R_{n, n'}^{\emr{2+\eps}}$, as it implies the other two by using Jensen inequality: for every $p\in [1,2+\eps]$ one has
    \[
    	\E R_{n,n'}^p = \E (R_{n,n'}^{2+\eps})^{\frac{p}{2+\eps}} \le (\E R_{n,n'}^{2+\eps})^{\frac{p}{2+\eps}} <
	(C_R^{2+\eps})^{\frac{p}{2+\eps}}  = C_R^p.
    \]

 Now, the random variable $R_{n,n'}$ has the form
 \[
 R_{n,n'} = \log \|B_1\|+\log \|B_2\|-\log \|B_2B_1\|,
 \]
 where $B_1= T_n$ and $B_2=T_{(n,n+n']}$.  Proposition~\ref{p:B1-B2} then implies
 \begin{equation}\label{eq:R-upper}
 R_{n,n'} \le \log 4\sqrt{2} + \log \min(\Delta_{T_n,T_{(n,n+n']}}^{-1} , \|T_{(n,n+n']}\|^2)
 \end{equation}

In what follows we will use the following regularity result from~\cite{M}. The setting of~\cite{M} is a general setting of non-stationary random dynamics: one assumes that a compact set~$\mKM$ of probability measures on $\Homeo(\Man)$ for some compact manifold~$\Man$ is given, and that these measures are concentrated on the bi-Lipschitz maps. It is also assumed that all these measures satisfy no deterministic image condition, and that the Lipschitz constant admits a uniformly bounded $\gamma$-th log-moment:
\[
\forall \mgr \in \mKM \quad \int_{\Lip(\Man)} (\log \max(\Lip(f),\Lip(f^{-1})))^{\gamma} \, \dd\mgr(f) < C_{\mK}.
\]
Then, the following estimates hold:
\begin{theorem}[{\cite[Theorem~B.1]{M}}]\label{thm:logHol}
Under the assumptions above, there exist constants $\kappa>1$ and $\CM$, such that
for any two initial probability measures $\msp^{(1)}_0$ and $\msp^{(2)}_0$ on~$\Man$ and any two sequences of iterations,
\[
\mgr_1^{(1)},\dots,\mgr_{n_1}^{(1)},\mgr_1^{(2)},\dots, \mgr_{n_2}^{(2)}\in\mK,
\]
for the random images
\begin{equation}\label{eq:random-images-msp}
\msp_1 =\mgr_{n_1}^{(1)}*\dots*\mgr_1^{(1)}*\msp^{(1)}_0, \quad \msp_2=\mgr_{n_2}^{(2)}*\dots*\mgr_1^{(2)}*\msp^{(2)}_0
\end{equation}
one has a uniform bound for $\gamma$-th log moment for the distance between the random points with a cut-off at radius $\theta_{n_1,n_2}:=\exp(-\kappa^{\min(n_1,n_2)})$:
\begin{equation}\label{eq:iint}
\iint_{M\times M} |\log \max (d(x,y), \theta_{n_1,n_2})|^{\gamma} \, \dd\nu_1 \dd\nu_2 < \CM.
\end{equation}
\end{theorem}
In particular, the Markov inequality for~\eqref{eq:iint} immediately implies a log-H\"older-type bound for the distances between these images
\begin{equation}\label{eq:log-Holder}
(\msp_1\times \msp_2) \{(x,y) \mid d(x,y)\le r \} <  \CM \cdot |\log r|^{-\gamma}
\end{equation}
for every $r\ge \theta_{n,n'}$.

In our case, $\Man$ will be the projective line $\RP^1$ and the maps will be the projectivisations of linear maps of~$\R^2$. Namely, to a probability measure $\mgr\in \mK$ on $\SL(2,R)$ we associate \emph{two} probability measures that are its pushforwards by the maps from $\SL(2,\R)$ to $\Homeo(\RP^1)$,
\[
F_1:A\mapsto f_A \quad \text{and} \quad F_2:A\mapsto f^{-1}_{A},
\]
and we let the set $\mKM$ be formed by these images,
\[
\mKM:=\{(F_1)_*\mgr \mid \mgr \in \mK\} \cup \{(F_2)_*\mgr \mid \mgr \in \mK\}.
\]
Then it satisfies the assumptions of Theorem~\ref{thm:logHol}: the set $\mKM$ is compact as a union of two continuous images of a compact~$\mK$, the moments assumption follows from the moments assumption for~$\mK$, and absence of a deterministic image follows from the similar condition for the measures from~$\mK$.

We will take initial measures
\[
\msp_0^{(1)}=\msp_0^{(2)}= \frac{1}{2} \delta_{[e_1]} +\frac{1}{2} \delta_{[e_2]}.
\]
Then, taking $n_1:=n$ and the sequence $\mgr_j^{(1)}:=(F_1)_* \mgr_j$, we see that the image measure $\msp_1$ is the half-sum of the laws of $[T_n e_1]$ and of~$[T_n e_2]$. In the same way, taking $n_2:=n'$ and the sequence
\[
\mgr_j^{(2)}:=(F_2)_* \mgr_{n+n'-(j-1)}, \quad j=1,\dots,n',
\]
we see that the image measure $\msp_2$ is the half-sum of the laws of $[T_{(n,n+n']}^{-1} e_1]$ and of~$[T_{(n,n+n']}^{-1} e_2]$. An application of Theorem~\ref{thm:logHol} thus yields a uniform upper bound
\[
\frac{1}{4} \sum_{i=1,2} \sum_{j=1,2} \E \left|\log \max \left\{ \dist\left([T_{n} e_i], [T_{(n,n+n']}^{-1} e_j]\right), \theta_{n,n'}\right\}\right|^{\gamma} \le \CM,
\]
and hence (\emr{replacing the sum over $i$ and $j$ with the maximal summand})
\[
\E \left|\log \max \left\{\Delta_{T_n,T_{(n,n+n']}}, \theta_{n,n'}\right\}\right|^{\gamma} \le 4\CM,
\]

Now, in the right hand side of~\eqref{eq:R-upper} we have a different maximum: instead of a cut-off at a fixed threshold~$\theta_{n,n'}$, the norm of the product $T_{(n,n+n']}$ appears.
\emr{To handle it, note that if $\Delta_{T_n,T_{(n,n+n']}}^{-1}\le \theta_{n,n'}^{-1}$, then
\[
 \min(\Delta_{T_n,T_{(n,n+n']}}^{-1} , \|T_{(n,n+n']}\|^2) \le \Delta_{T_n,T_{(n,n+n']}}^{-1} = \min (\Delta_{T_n,T_{(n,n+n']}}^{-1}, \theta_{n,n'}^{-1}),
\]
while if $\Delta_{T_n,T_{(n,n+n']}}^{-1}> \theta_{n,n'}^{-1}$, then
\[
 \min(\Delta_{T_n,T_{(n,n+n']}}^{-1} , \|T_{(n,n+n']}\|^2) \le \|T_{(n,n+n']}\|^2 ;
\]
applying the logarithm and joining these two inequalities, we get
}
\begin{multline}\label{eq:splitting}
 \log \min(\Delta_{T_n,T_{(n,n+n']}}^{-1} , \|T_{(n,n+n']}\|^2) \le  \left|\log \max (\Delta_{T_n,T_{(n,n+n']}} , \theta_{n,n'})\right|  +
 \\
 + (\log \|T_{(n,n+n']}\|^2) \cdot \Ind_{\{\Delta_{T_n,T_{(n,n+n']}}< \theta_{n,n'}\}}.
\end{multline}
The $\left(2+\eps\right)$-th moment of second summand can be then estimated by the H\"older inequality with the exponents $p=\frac{\gamma}{2+\eps}>1$ and $q=\frac{1}{1-\frac{1}{p}}$:
\begin{multline}\label{eq:finalizing}
\E \left[ (2\log \|T_{(n,n+n']}\|)^{2+\eps} \cdot \Ind_{\{\Delta_{T_n,T_{(n,n+n']}}< \theta_{n,n'}\}} \right]
\\
\le  \left(\E (2\log \|T_{(n,n+n']}\|)^{\gamma}\right)^{1/p} \cdot \left( \Prob \{\Delta_{T_n,T_{(n,n+n']}}< \theta_{n,n'}\}\right)^{1/q}
\\
\le (2n')^{2+\eps} M^{1/p} \cdot (4\CM \cdot |\log \theta_{n,n'}|^{-\gamma})^{1/q}
\end{multline}
Here we have noticed that $\log \|T_{(n,n+n']}\|$ doesn't exceed a sum of $n'$ summands of the form $\log \|A_j\|$, the $\gamma$-th moment of each of which does not exceed~$M$, and used the Markov inequality~\eqref{eq:log-Holder} to estimate the probability in the second factor.  Now, as $|\log \theta_{n,n'}|=\kappa^{\min(n,n')}$ grows \emph{exponentially} in $n'$ (recall that $n\le 2n'$ due to the assumption of the proposition), the right hand side of~\eqref{eq:finalizing} is uniformly bounded (and actually tends to~$0$ as $n'$ grows).

We have shown that both random variables in the right hand side of~\eqref{eq:splitting} have uniformly bounded $\left(2+\eps\right)$-th moments, and hence the same applies to~$R_{n,n'}$: the upper bound~\eqref{eq:R-upper} for it differs by an addition of a constant~$\log 4\sqrt{2}$.
\end{proof}

%\newpage

%\vspace{2cm}
%\newpage

\section{Moments growth for $\xi_n$}\label{s:xi-moments}
\subsection{Probabilistic setup}

In this section, we will obtain the bounds for the moments of the centred random variables $\txi_n:=\xi_{n}-\E \xi_n$. The upper estimates will be obtained by purely probabilistic methods: we do not consider the geometry of the problem, using only the estimates obtained in Proposition~\ref{prop:Rmoments} for the moments of differences $R_{n,n'}$.

Namely, assume that we are given a family of random variables $\{\txi_{a;n}\}_{a\ge 0, \, n\ge 1}$; \emr{the link with the products of random matrices will be given by setting
\begin{equation}\label{eq:txi-def}
\txi_{a;n}:= \xi_{(a,a+n]}-  \E \xi_{(a,a+n]}.
\end{equation}
Note that the second index in this new family corresponds to the ``\emph{length}'' of the product, and not to the last index; this choice is made due to the importance of this length (the estimates will be uniform in $a$ as $n$ tends to infinity).
}

Define the difference random variables
\begin{equation}\label{eq:tR-def}
\tR_{a;n,n'}:= \txi_{a;n} + \txi_{a+n;n'} -\txi_{a;n+n'},
\end{equation}
and assume that the following assumptions hold:
\begin{enumerate}[(a)]
\item\label{a:center} \textbf{Centering:} $\E\txi_{a;n}=0$ for all $a,n$.
\item\label{a:independence} \textbf{Independence}: for any $a,a',n,n'$ such that $a'\ge a+n$, the random variables $\txi_{a;n}$ and $\txi_{a',n'}$ are independent.
\item\label{a:initial} \textbf{Initial moments:} The variables $\xi_{a,n}$ have finite $2+\eps$-th moments, and
\begin{equation}
\exists C_M: \quad \forall a=0,1,2,\dots, \quad \E|\txi_{a;1}|^{2+\eps}<C_M^{2+\eps}.
\end{equation}
\item\label{a:difference} \textbf{Moments of differences:} There exists $C_R$ such that for any $a$ and any $n,n'$ with $\frac{n'}{2}\le n \le 2n'$, one has
\begin{equation}\label{eq:E-tR}
\E |\tR_{a;n,n'}|^{2+\eps} \le (2C_R)^{2+\eps}.
\end{equation}
\end{enumerate}

\emr{The link~\eqref{eq:txi-def} to our setting of products of random matrices is then provided by the following lemma. }
\begin{lemma}\label{l:txi-assumptions}
Under the assumptions of Theorem~\ref{t.CLT}, the family of random variables \emr{$\{\xi_{a,n}\}$, defined by~\eqref{eq:txi-def},}
%\begin{equation}\label{eq:txi-def}
%\txi_{a;n}:= \xi_{(a,a+n]}-  \E \xi_{(a,a+n]}
%\end{equation}
satisfies the assumptions~(\ref{a:center})--(\ref{a:difference}) above.
\end{lemma}
\begin{proof}
The only non-trivial conclusion is the estimate~(\ref{a:difference}). To ensure it, recall that Proposition~\ref{prop:Rmoments} guarantees for
$\frac{n'}{2}\le n\le 2n'$ the bound $\E R_{a;n,n'}^{2+\eps}\le C_R^{2+\eps}$, where
\[
R_{a;n,n'}:=\xi_{(a,a+n]}+\xi_{(a+n,a+n+n']}- \xi_{(a,a+n+n']} \ge 0.
\]
Now, this (due to the Jensen's inequality for $x^{1/(2+\eps)}$) implies $\E R_{a;n,n'}\le C_R$, and hence the random variable~$\emr{\tR_{a;n,n'}=R_{a;n,n'}-\E R_{a;n,n'}}$ is the difference between two random variables, whose $L_{2+\eps}(\P)$-norm doesn't exceed $C_R$. The desired~\eqref{eq:E-tR} then follows from the triangle inequality for~$L_{2+\eps}$-norm.
\end{proof}

In the above probability setting, the following estimate holds; it will be proved below, in Section~\ref{s:xi-moments-upper}.

\begin{prop}\label{prop:ximoments}
	Assume that the family $\txi_{a;n}$ satisfies the assumptions (\ref{a:center})--(\ref{a:difference}) above. Then, there exists a constant $\Cxi$ such that for all $a$ and $n$,
	\begin{equation}\label{eq:Moment-2-eps}
	        \E |\txi_{a;n}|^{2+\eps} < \Cxi^{2+\eps} n^{1+\frac{\eps}{2}}.
	\end{equation}
	As a corollary, for any $p\in [0,2+\eps]$, any $a$ and $n$, one has
	\begin{equation}\label{eq:Moment-p}
	        \E |\txi_{a;n}|^{p} < \Cxi^{p} n^{p/2}.
	\end{equation}
\end{prop}

Note that in the case of $\txi_{a;n}$ being sums of independent $\mN(0,1)$-summands (and in particular, all $R_{a;n,n'}=0$), one would have $\txi_{a;n}\sim \mN(0,n)$, and thus the $p$-th moment of $\txi_{a;n}$ would scale exactly as~$n^{p/2}$. The conclusion of Proposition~\ref{prop:ximoments} extends upper bounds with such a scaling to the more general setting of the assumptions above.

Next, we will need a lower growth bound for the variances of~$\txi_{a;n}$, that would grow linearly in~$n$. For that, we will need an extra assumption for the sequence~$\txi_{a;n}$:

\begin{enumerate}[(a)]\addtocounter{enumi}{4}
\item\label{a:growth} \textbf{Variation growth}: For every $c$ there exists $n_1\in \mathbb{N}$ such that
\begin{equation}\label{eq:txi-c}
\forall n\ge n_1 \quad \forall a \quad
\var \txi_{a;n}\ge c.
\end{equation}
\end{enumerate}

Again, this assumption is guaranteed for our matrices-related setting: it is the result of the following proposition.

%\begin{remark}
%in order to apply the arguments below to our case, we will apply them to $\txi_{a;n}:= \xi_{(a,a+n]}-  \E \xi_{(a,a+n]}$.
%\end{remark}

\begin{prop}\label{p:large}
    Under the assumptions of Theorem \ref{t.CLT}, for any $c>0$ there exists $n_1\in \mathbb{N}$ such that for any $n\ge n_1$ and any collection of distributions $\mgr_1,\dots,\mgr_n\in \mathcal{K}$ one has
    \[
    \var_{\mgr_1,\dots,\mgr_n} \xi_{n} \ge c.
    \]
\end{prop}

We will provide the proof of Proposition \ref{p:large} in Section \ref{section:variance}. Now, using this extra assumption, we get the following estimate:
%~(\ref{a:growth}) for the family $\txi_{a;n}$
% coming from this proposition

\begin{lemma}\label{l:var-lower}
Under the assumptions~(\ref{a:center})--(\ref{a:growth}) above, there exists $C_1 > 0$ and $n_0$ such that
    \begin{equation}
    	\forall a \quad \forall n\ge n_0 \quad \var \txi_{a;n} \ge C_1^2 n.
    \end{equation}
\end{lemma}
We will prove Lemma~\ref{l:var-lower} in Section~\ref{s:linear-var} below, using the same ideas of controlling the influence of the discrepancy $R_{a;n,n'}$ as those that appear in the proof of Proposition~\ref{prop:ximoments}. We will then proceed in Section~\ref{s.Distance} to the proof of the following result, that immediately implies Theorem~\ref{t.CLT}.

\begin{theorem}\label{t:proba}
Let the family $\txi_{a;n}$ satisfy the assumptions~(\ref{a:center})--(\ref{a:growth}). Then there exist constants $C_1,\Cxi$ and a number $n_0$ such that for every $n>n_0$
\begin{equation}\label{eq:var-lin-growth}
\forall a \quad C_1^2\, n \le \var \txi_{a;n} \le \Cxi^2 \, n,
\end{equation}
and random variables
\begin{equation}
\frac{\txi_{0;n}}{\sqrt{\var \txi_{0;n}}}
\end{equation}
converge to $\mN(0,1)$ in distribution as $n\to\infty$.
\end{theorem}

\subsection{Upper bounds on moments of~$\txi_{a;n}$}\label{s:xi-moments-upper}

\begin{proof}[Proof of Proposition~\ref{prop:ximoments}]
We will start by establishing the conclusion \eqref{eq:Moment-p} for $p=2$ (and thus for all smaller values of~$p$), in other words, a linear upper bound for the growth of the variance of $\txi_{a;n}$:
\begin{equation}\label{eq:C-2}
\exists \CVar: \quad \forall a, \quad \forall n \quad \var \txi_{a;n}\le \CVar^2 \cdot n.
\end{equation}

To do so, we will recurrently define the sequence $c_m$, such that for each $m=1,2,\dots$ one has
\begin{equation}\label{eq:var-c-n}
\forall a \quad \var \txi_{a;m}\le c_m^2 \cdot m.
\end{equation}
Namely, we start by setting $c_1:=C_M$; the estimate~\eqref{eq:var-c-n} is then satisfied due to initial moments assumption~(\ref{a:initial}). Now, to construct $c_m$ with $m>1$, take $n = \left\lfloor \frac{m}{2} \right\rfloor $ and $n' = m - n$.
Then for any $a$,
\[
\var (\txi_{a;n} + \txi_{a+n;n'}) = \var \txi_{a;n} + \var \txi_{a+n;n'} \le c_n^2 \cdot n + c_{n'}^2 \cdot n'\le  \left(\max(c_n,c_{n'})\right)^2 \cdot m.
\]
Now, the $L_2$-norm triangle inequality (or Cauchy-Schwarz inequality)
\[
\sqrt{\var (X-Y)} \le \sqrt{\var X} + \sqrt{\var Y}
\]
holds for any two random variables $X,Y$ with finite second moment; applying it for
\[
\txi_{a;m}=\underbrace{(\txi_{a;n} + \txi_{a+n;n'})}_{X}-\underbrace{\tR_{a;n,n'}}_Y,
\]
and using assumption~(\ref{a:difference}) that implies $\sqrt{\var \tR_{a;n,n'}} \le 2 C_R$, one gets
\begin{equation}\label{eq:var-new}
\frac{1}{\sqrt{m}} \sqrt{\var \txi_{a;m}} \le \max(c_n,c_{n'}) + \frac{2 C_R}{\sqrt{m}}.
\end{equation}
Hence, it suffices to take $c_m$ to be the right hand side of~\eqref{eq:var-new}:
\begin{equation}\label{eq:c-m-recurrence}
c_m:=\max(c_n,c_{n'}) + \frac{2 C_R}{\sqrt{m}},
\end{equation}
to ensure that~\eqref{eq:var-c-n} holds for this~$m$. For the sequence defined by the recurrence relation~\eqref{eq:c-m-recurrence}, it is easy to check by induction that for all $m=2^k+1, \dots, 2^{k+1}$ we have
\[
	c_m \le c_1  + \sum_{j=0}^k \frac{2C_R}{\sqrt{2^j}},
\]
which in turn  implies a uniform bound
\[
	c_m \le C_M + \frac{2C_R}{1-\frac{1}{\sqrt{2}}}=:\CVar,
\]
thus concluding the proof of~\eqref{eq:C-2}. Note also that due to the Jensen inequality, the established estimate also implies
\begin{equation}\label{eq:C-2-p}
\forall a, \quad \forall n \quad \E |\txi_{a;n}|^p\le (\CVar \cdot \sqrt{n})^p
\end{equation}
for every $p\in [0,2]$.

Let us now pass to the estimate of the $2+\eps$-th moments. We are going again to define recurrently a sequence $c'_m$, such that for all $m=1,2,\dots$ one has
\begin{equation}\label{eq:c-prim}
\forall a \quad \E |\txi_{a;m}|^{2+\eps} \le (c'_m\cdot \sqrt{m})^{2+\eps}.
\end{equation}
As before, we start by setting $c'_1:=C_M$; the estimate~\eqref{eq:c-prim} is then satisfied due to initial moments assumption~(\ref{a:initial}). Now, to construct $c_m$ with $m>1$, we take $n = \left\lfloor \frac{m}{2} \right\rfloor $ and $n' = m - n$.
We will use the following lemma (postponing its proof until the end of the proof of Proposition~\ref{prop:ximoments}):
\begin{lemma}\label{l:2+eps}
Let $X$, $Y$ be two independent random variables. Then
\begin{equation}\label{eq:2+eps}
\E |X+Y|^{2+\eps} \le \E |X|^{2+\eps}+ \E |Y|^{2+\eps} + 2^{2+\eps} \left(\E |X| \cdot \E |Y|^{1+\eps}+ \E |X|^{1+\eps} \cdot \E |Y|\right).
\end{equation}
\end {lemma}

Now, let us apply Lemma~\ref{l:2+eps} to $X=\txi_{a;n}$ and $Y=\txi_{a+n;n'}$. Note that due to the upper bound~\eqref{eq:C-2-p} for first and $1+\eps$-th moments, we get
\[
\E |X|^{1+\eps} \cdot \E|Y| \le \CVar^{1+\eps}\, \sqrt{n}^{1+\eps}\cdot \CVar \, \sqrt{n'} \le \CVar^{2+\eps} \sqrt{m}^{2+\eps},
\]
and the same estimate applies to $\E |X|\cdot \E|Y|^{1+\eps}$.
From the estimate~\eqref{eq:2+eps} we thus get
\begin{multline}\label{eq:txi-2-eps}
\E \left| \txi_{a;n}+\txi_{a+n;n'}\right|^{2+\eps} \le \E \left| \txi_{a;n}\right|^{2+\eps} +\E \left| \txi_{a+n;n'}\right|^{2+\eps} + 2 (2\CVar)^{2+\eps} \cdot \sqrt{m}^{2+\eps}
\\
\le (\max(c'_n, c'_{n'}))^{2+\eps}\cdot \left( n^{1+\frac{\eps}{2}} + (n')^{1+\frac{\eps}{2}} \right)+ (4\CVar)^{2+\eps} \cdot m^{1+\frac{\eps}{2}}
%\\
%\le \left[ (\max(c'_n, c'_{n'}))^{2+\eps}\cdot
%\left( \left(\frac{n}{m}\right)^{1+\frac{\eps}{2}} + \left(\frac{n'}{m}\right)^{1+\frac{\eps}{2}} \right)+ (4\CVar)^{2+\eps}\right]
%\cdot m^{1+\frac{\eps}{2}}
\end{multline}
Now, let
\[
\lambda_m:= \left(\frac{n}{m}\right)^{1+\frac{\eps}{2}} + \left(\frac{n'}{m}\right)^{1+\frac{\eps}{2}}.
\]
%Dividing~\eqref{eq:txi-2-eps} by $m^{1+\frac{\eps}{2}}$,
Taking $2+\eps$-th power root, and using the concavity inequality $\sqrt[2+\eps]{a+b}\le \sqrt[2+\eps]{a}+\sqrt[2+\eps]{b}$, we obtain
\begin{equation}
\left( \E \left| \txi_{a;n}+\txi_{a+n;n'}\right|^{2+\eps}\right)^{1/(2+\eps)} \le (\lambda_m^{1/(2+\eps)} \cdot \max(c'_n, c'_{n'}) + 4\CVar) \cdot \sqrt{m}.
\end{equation}
Now, using the $L_{2+\eps}$--triangle inequality for $(\txi_{a;n} + \txi_{a+n;n'})$ and $\tR_{a;n,n'}$, we get
\[
\left( \E \left| \txi_{a;m}\right|^{2+\eps}\right)^{1/(2+\eps)} \le  (\lambda_m^{1/(2+\eps)} \cdot \max(c'_n, c'_{n'}) + 4\CVar) \cdot \sqrt{m} + 2C_R.
\]
Hence, it suffices to take
\begin{equation}\label{eq:c-prim-m}
c'_m:= \lambda_m^{1/(2+\eps)} \cdot \max(c'_n, c'_{n'}) + 4\CVar + \frac{2C_R}{\sqrt{m}}
\end{equation}
for the upper bound $\E\left| \txi_{a;m}\right|^{2+\eps} \le (c'_m \cdot \sqrt{m})^{2+\eps}$ to hold.

Now, for every $m$ one has $\lambda_m<1$, and the values $\lambda_m$ converge to $2^{-\frac{\eps}{2}}<1$ as $m\to\infty$. Hence, there exists $\lambda<1$ such that for all $m$, one has $\lambda_m^{1/(2+\eps)}<\lambda$. Taking
\[
\Cxi:=\frac{4\CVar+2C_R}{1-\lambda}
\]
to be the fixed point of the map $c\mapsto \lambda c +(4\CVar+2C_R)$, we see by recurrence that the sequence~$(c'_m)$ satisfies $c'_m\le \Cxi$ for all~$m$. We have thus obtained the desired upper bound
\[
\E |\txi_{a;m}|^{2+\eps} \le \Cxi^{2+\eps} \cdot \sqrt{m}^{2+\eps}.
\]

Let us now prove Lemma~\ref{l:2+eps}:
\begin{proof}[Proof of Lemma~\ref{l:2+eps}]
Let us first show that for any $a,b\in \R$,
\begin{equation}\label{eq:a-b}
|a+b|^{2+\eps} \le |a|^{2+\eps} + |b|^{2+\eps}+ 2^{2+\eps}\cdot (|a|^{1+\eps} |b| +|a| \cdot |b|^{1+\eps}).
\end{equation}
Indeed, Jensen's inequality for the function $(1+x)^{2+\eps}$ on $[0,1]$ implies that
\begin{equation}\label{eq:J-x}
\forall x\in [0,1] \quad (1+x)^{2+\eps} \le (1-x)\cdot 1 + x\cdot 2^{2+\eps} \le 1 + 2^{2+\eps}x.
\end{equation}
Now, if $|a|\ge |b|$, then
\begin{equation}\label{eq:a-b-1}
|a+b|^{2+\eps} \le |a|^{2+\eps} \cdot \left(1+\frac{|b|}{|a|}\right)^{2+\eps} \le |a|^{2+\eps} + 2^{2+\eps}\cdot |a|^{1+\eps} |b|,
\end{equation}
where we have applied~\eqref{eq:J-x} with $x=\frac{|b|}{|a|}$. In the same way, if $|b|\ge |a|$, then
\begin{equation}\label{eq:a-b-2}
|a+b|^{2+\eps} \le |b|^{2+\eps} + 2^{2+\eps} \cdot |b|^{1+\eps} |a|.
\end{equation}
Taking the sum of the right hand sides of~\eqref{eq:a-b-1} and~\eqref{eq:a-b-2}, we obtain the desired upper bound~\eqref{eq:a-b}. It now suffices to apply this inequality for $X$ and $Y$, and to take the expectation.
\end{proof}
With Lemma~\ref{l:2+eps} proven, the proof of Proposition~\ref{prop:ximoments} is now complete.
\end{proof}

\subsection{Linear growth of variances}\label{s:linear-var}

This section is devoted to the proof of Lemma~\ref{l:var-lower}; together with already established Proposition~\ref{prop:ximoments}, it implies the linear growth conclusion~\eqref{eq:var-lin-growth} of Theorem~\ref{t:proba}, and thus inequality~\eqref{VarGrowth} in Theorem~\ref{t.CLT}.

\begin{proof}[Proof of Lemma~\ref{l:var-lower}]
    Recall that
    \begin{equation*}%\label{eq:txi-a-n-np}
         \txi_{a;n + n'} = (\txi_{a;n} + \txi_{a+n; n'}) - \tR_{a;n, n'};
    \end{equation*}
    Cauchy-Schwarz inequality thus implies that
    \begin{equation}\label{eq:Cauchy-lower}
        \sqrt{\var \txi_{a;n + n'}} \ge \sqrt{\var (\txi_{a;n} + \txi_{a+n; n'}) } - \sqrt{\var \tR_{a;n,n'}}.
    \end{equation}

\emr{We will use~\eqref{eq:Cauchy-lower} to establish by induction that for some $n_0$ and $q>0$ for all $m\ge n_0$ and all $a$ one has
	\begin{equation}\label{eq:var-lower}
        \sqrt{\var \txi_{a;m}} \ge \sqrt{q (m+1)}  + 3 C_R.  % (n+1)
	\end{equation}
	To do that, set $n_0$ to be equal to $n_1$ for which~\eqref{eq:txi-c} holds with $c=16C_R^2$, and let $q:=\frac{C_R^2}{2n_0}$, where $C_R$ is given by the assumption~(\ref{a:difference}). This choice guarantees that~\eqref{eq:var-lower} holds for any $a$ and every $m=n_0,\dots,2n_0-1$: we have
	\[
		\sqrt{\var \txi_{a;m}} \ge \sqrt{16C_R^2} = 4 C_R \ge \sqrt{\frac{C_R^2}{2n_0} \cdot (m+1)} + 3C_R.
	\]
}

    \emr{Now, we proceed by induction to show that~\eqref{eq:var-lower} actually holds for every $m\ge n_0$.}
    Indeed, let $m\ge 2n_0$ be the first number for which~\eqref{eq:var-lower} \emr{has not yet been} established; decompose it as $m=n+n'$, where $n=\lfloor \frac{m}{2} \rfloor$, $n'=\lceil \frac{m}{2} \rceil=m-n$.
    Then, each of the variances $\var \txi_{a;n}$, $\var \txi_{a+n;n'}$ in~\eqref{eq:Cauchy-lower} is bounded from below by $\left(\sqrt{q (n +1)} + 3C_R\right)^2$, and hence
    \begin{multline*}
        \sqrt{\var \txi_{a;n}+ \var \txi_{a+n;n'} } - \sqrt{\var R_{a;n,n'}} \ge
        \\
         \ge \sqrt{2} \cdot (\sqrt{q (n +1)} + 3C_R) - C_R \ge
         \\
        \ge \sqrt{q (m+1)} + (3\sqrt{2}-1) C_R,
    \end{multline*}
    where we have used $2 (n +1) \ge m+1$, as $n=\lfloor \frac{m}{2} \rfloor$. As $3\sqrt{2}-1>3$, this proves the induction step. In particular, for every $m\ge n_0$ and any $a$ we have $\var \txi_{a;m} \ge q m$, that proves the lemma with $C_1=\sqrt{q}=\frac{C_R}{\sqrt{2n_0}}$.

\end{proof}

%Proposition \ref{p:large} and Lemma \ref{} together prove \eqref{VarGrowth}.

\emr{Now, define the normalised family
\begin{equation}\label{eq:eta-family}
\eta_{a;n}:=\frac{1}{\sqrt{\var \txi_{a;n}}}\,\txi_{a;n}.
\end{equation}
Combining the upper bound for the $2+\eps$-th moment of $\txi_{a;n}$ by $(\Cxi\sqrt{n})^{2+\eps}$ from Proposition~\ref{prop:ximoments} with a lower bound for its variance by $(C_1 \sqrt{n})^2$ from Lemma~\ref{l:var-lower}, we get the following uniform upper bound.}

\begin{coro}\label{c.3}
Under the assumptions~(\ref{a:center})--(\ref{a:growth}) above, there exists $\bC$, such that for all $n\ge n_0$, \emr{with $n_0$ given by Lemma~\ref{l:var-lower}}, and any $a$, the normalised variable~$\eta_{a;n}$, defined by~\eqref{eq:eta-family}, satisfies
    \[
        \E |\eta_{a;n}|^{2+\eps} < \bC^{\emr{2+\eps}}.
    \]

\end{coro}
\begin{proof}
\emr{Indeed, one has
    \[
        \E |\eta_{a;n}|^{2+\eps} < \frac{(\Cxi \sqrt{n})^{2+\eps}}{(C_1 \sqrt{n})^{2+\eps}} = \frac{\Cxi^{2+\eps}}{C_1^{2+\eps}},
    \]
    hence it suffices to take $\bC:=\Cxi/C_1$.
    }
\end{proof}

%\newpage

\section{Bootstrapping: distance to the Gaussian distribution}\label{s.Distance}
This section is devoted to the bootstrapping arguments that allow to show the convergence to Gaussian law.

\subsection{Preliminaries}

 Let $\xi$, $\xi'$ be two independent random variables with finite \emr{$\left(2+\eps\right)$-th}
moments and comparable variances: for a given constant $C>1$, we have
\[
C^{-1}<\frac{\var \xi}{\var \xi'} < C.
\]
We will provide a value $N'_{\rad}(\xi)$, measuring  quantitatively ``non-Gaussianity'' of the law of $\xi$, such that (under appropriate assumptions) it will be smaller for the sum $\xi+\xi'$ than for the summands separately.
%with zero mean and unit variance; we'll be applying

\begin{defi}
    For a random variable $\xi$ we denote by $\varphi_{\xi} (t)$ its \emph{characteristic function}:
    \begin{equation*}
        \varphi_{\xi}(t) = \mathbb{E} e^{it\xi}.
    \end{equation*}
    Now, let
    \begin{equation}\label{e.Nr}
        N_{\rad}(\eta) = \sup_{0< |t|< \rad}  \frac{ \left|\log \left( \varphi_{\eta} (t) e^{t^2 / 2} \right) \right|}{|t|^{\emr{2+\eps}}} , \quad N_{\rad}'(\xi)= N_{\rad}\left(\frac{\xi-\E \xi}{\sqrt{\var \xi}}\right).
    \end{equation}
\end{defi}

Then $\eta\sim \mathcal{N}(0,1)$ if and only if $N_{\rad}(\eta)=0$ for all~$\rad>0$ (as the distribution of a random variable is uniquely determined by its characteristic function). Note also that $N_{\rad}(\eta)$ might be infinite if the corresponding characteristic function $\varphi_{\eta}$ vanishes somewhere on $[-\rad,\rad]$. Finally, the logarithm here is a function of a complex variable (as the characteristic function might be, and most often is, non-real). As soon as~$\varphi_{\eta}(t)$ doesn't vanish on $[-\rad,\rad]$, we define the composition $\log \left( \varphi_{\eta} (t) e^{t^2 / 2} \right)$ by a continuous extension, starting with the value $\log 1=0$ at $t=0$.

\subsection{Initial estimates}\label{s:initial}
To start a bootstrapping argument, one needs some initial bounds, in this case, for the norms $N_{\rad}'(\xi)$ for some $\rad>0$. The first step to obtain these is the following bound for the characteristic functions:
%; we will apply these for random variables~$\txi_{a;1}$.

\begin{lemma}\label{l.phi-3}
    Let $X$ be a random variable with
    \[
    \E X=0, \quad \var X =1, \quad \E |X|^{\emr{2+\eps}} < C_X^{\emr{2+\eps}}.
    \]
    Then its characteristic function satisfies
    \begin{equation}\label{eq.X}
        \quad \left|\varphi_X(t)-\left(1-\frac{t^2}{2}\right)\right| \le C_X^{\emr{2+\eps}} \cdot |t|^{\emr{2+\eps}} \ \ \text{for all}\ t\in \mathbb{R}.
    \end{equation}
\end{lemma}
\begin{proof}
    Note that it suffices to establish the \emr{following} estimate for the second derivative~$\varphi''(t)$: for all $t\in\R$,
    \begin{equation}\label{eq.ddphi}
        |\varphi_X''(t) + 1| \le \emr{2} C_X^{\emr{2+\eps}} \, |t|^{\emr{\eps}}.
    \end{equation}
    Indeed, integrating~\eqref{eq.ddphi} two times then suffices to obtain~\eqref{eq.X}:
    \[
    \varphi_{X}(t)- \left(1-\frac{t^2}{2}\right) = \int_0^t \, dt_1 \int_0^{t_1} (\varphi_{X}''(t_2) +1) \, dt_2.
    \]

    Now, let us rewrite the estimated expression in~\eqref{eq.ddphi}:
    \begin{equation}\label{eq.ddphi-E}
    |\varphi_X''(t)+1| = | \E (X^2 (e^{itX}-1)) | \le \E (X^2 |e^{itX}-1|) \le \E (X^2 \cdot \emr{2} |tX|^{\emr{\eps}});
    \end{equation}
    here we have used $\var X=1$ for the first equality, and \emr{the upper bound}
    \begin{equation}\label{eq:exp-ix-bound}
	    \emr{ |e^{ix}-1| \le \min(2, |x|) \le 2|x|^{\eps}  }
    \end{equation}
     for the last inequality. Finally, the right hand side of~\eqref{eq.ddphi-E} can be rewritten as
    \[
    \emr{2} |t|^{\emr{\eps}}\cdot \E |X|^{\emr{2+\eps}} < \emr{2} C_X^{\emr{2+\eps}} |t|^{\emr{\eps}},
    \]
    thus completing the proof.
\end{proof}

Joining this lemma with the estimate from Corollary~\ref{c.3}, we get an initial bound lemma for the family~$\{\txi_{a;n}\}$.

\begin{lemma} \label{lm:Start}
Under the assumptions~(\ref{a:center})--(\ref{a:growth}) above, for the normalised variables~$\eta_{a;n}$, defined by~\eqref{eq:eta-family}, the following estimate holds. There exists $\rad_0 > 0$, such that for any $n\ge n_0$, \emr{where $n_0$ is given by Lemma~\ref{l:var-lower}}, and any~$a$ the value $N_{\rad_0} (\eta_{a;n})$ is well-defined and satisfies
    \begin{equation*}
        N_{\rad_0} (\eta_{a;n}) < 3 \bC^{\emr{2+\eps}} \quad \text{and} \quad \rad_0^{\emr{2+\eps}} N_{\rad_0} (\eta_{a;n}) < \frac{1}{100},
    \end{equation*}
    where $\bC$ is the constant defined in Corollary~\ref{c.3}.

    Moreover, the choice of $\rho_0$ depends only on the constants under the assumptions~(\ref{a:center})--(\ref{a:growth}).
\end{lemma}
\begin{proof}
    Applying Lemma~\ref{l.phi-3}, and multiplying~\eqref{eq.X} by $e^{\frac{t^2}{2}}$, we get
    \begin{equation}\label{eq.C3}
        \left| e^{\frac{t^2}{2}} \varphi_{\eta_{a;n}}(t) - e^{\frac{t^2}{2}} \left(1-\frac{t^2}{2}\right) \right| \le \bC^{\emr{2+\eps}} |t|^{\emr{2+\eps}} \cdot e^{\frac{t^2}{2}}
    \end{equation}
    Now,
    \[
        e^{\frac{t^2}{2}} \left(1-\frac{t^2}{2}\right) = 1 + o(|t|^{\emr{2+\eps}}),
    \]
    thus for sufficiently small $\rad$ one has
    \begin{equation}\label{eq.exp-1}
        \forall |t|\le \rad \qquad \left|e^{\frac{t^2}{2}} \left(1-\frac{t^2}{2}\right) - 1 \right| \le \frac{1}{2} \bC^{\emr{2+\eps}} |t|^{\emr{2+\eps}},
    \end{equation}
    as well as
    \[
        \forall |t|\le \rad \qquad e^{\frac{t^2}{2}} \le e^{\frac{\rad^2}{2}} < \frac{3}{2},
    \]
    hence the right hand side of~\eqref{eq.C3} can be replaced by~$\frac{3}{2} \bC^{\emr{2+\eps}} |t|^{\emr{2+\eps}}$. Joining this with~\eqref{eq.exp-1}, we get
    \[
        \forall |t|\le \rad \qquad \left| e^{\frac{t^2}{2}} \varphi_{\eta_{a;n}}(t) - 1 \right| \le \left(\frac{1}{2}+\frac{3}{2}\right) \cdot  \bC^{\emr{2+\eps}} |t|^{\emr{2+\eps}} = 2 (\bC\cdot |t|)^{\emr{2+\eps}}.
    \]
    Taking
    \[
    \rad_0:=\min\left(\rad, \frac{1}{20 \bC}\right),
    \]
    we ensure that $2(\bC \rad_0)^{\emr{2+\eps}}<\frac{1}{100}$. Hence the function $\log z$ is $\frac{3}{2}$-Lipschitz in the (complex) disc~$U_{2\bC \rad_0^{\emr{2+\eps}}}(1)$, and therefore,
    \[
        \forall |t|\le \rad_0 \quad \left|\log  \left(e^{\frac{t^2}{2}} \varphi_{\eta_{a;n}}(t) \right) \right|\le \frac{3}{2} \cdot 2\bC^{\emr{2+\eps}} |t|^{\emr{2+\eps}},
    \]
    which implies the desired conclusions
    \[
        N_{\rad_0}(\eta_{a;n}) < 3 \bC^{\emr{2+\eps}} \quad \text{ and } \quad \rad_0^{\emr{2+\eps}} N_{\rad_0} (\eta_{a;n}) < \frac{1}{(20 \bC)^{\emr{2+\eps}}} \cdot 3\, \bC^{\emr{2+\eps}} < \frac{1}{100}.
    \]
%    The second part of the conclusion follows from the choice of~$\rad_0$.
\end{proof}

\subsection{Sum of two independent variables}\label{s:sum-independent}

The following is the first step of the bootstrapping argument, estimating the decrease of $N'$-values for the sum of two independent random variables. Notice that in addition to the decrease by a linear factor, the parameter $\rad$ (describing the size of the domain) gets increased.
\begin{lemma}\label{lm:Contr}
    For any $C>0$ there exists $\lambda<1$ and $L>1$ such that if for some $\rad>0$ for some independent random variables $\xi,\xi'$ one has
    \[
        C^{-1}<\frac{\var \xi}{\var \xi'}<C
    \]
    and values $N'_{\rad}(\xi),N'_{\rad}(\xi')$ are finite, then
    \begin{equation}\label{eq:N-xi-xi-prim}
        N'_{L\rad}(\xi+\xi') \le \lambda \cdot
        \max( N'_{\rad}(\xi), N'_{\rad}(\xi') ).
    \end{equation}
\end{lemma}
%(sf. Lemma~\ref{lm:Contr})
\begin{proof}
Let
\[
    \eta=\frac{\xi-\E\xi}{\sqrt{\var \xi}}, \quad \eta'=\frac{\xi'-\E\xi'}{\sqrt{\var \xi'}}, \quad
    \eta''=\frac{(\xi+\xi')-(\E(\xi+\xi'))}{\sqrt{\var (\xi+ \xi')}},
\]
Also, denote
\[
    c=\sqrt{\frac{\var \xi}{\var\xi + \var \xi'}}, \quad
    c'=\sqrt{\frac{\var \xi'}{\var\xi + \var \xi'}};
\]
then, one has
\[
    \eta''= c \eta + c'\eta',
\]
with the coefficients that satisfy
\[
    c^2+(c')^2=1, \quad c,c'\le \sqrt{\frac{C}{C+1}}<1.
\]
By definition, we have for any $L$
\[
    N'_{L\rad}(\xi+\xi') = N_{L\rad}(\eta'')=N_{L\rad}(c \eta + c'\eta').
\]
Now, for the characteristic functions we have
\[
    \varphi_{c\eta+c'\eta'}(t) = \varphi_{\eta}(ct)\cdot \varphi_{\eta'}(c't),
\]
and as $c^2+(c')^2=1$, we have
\begin{equation}\label{eq:t-cc}
    e^{\emr{\frac{1}{2}} t^2}\varphi_{c\eta+c'\eta'}(t) = e^{\emr{\frac{1}{2}}(ct)^2}\varphi_{\eta}(ct)\cdot
    e^{\emr{\frac{1}{2}}(c't)^2}\varphi_{\eta'}(c't).
\end{equation}
%Hence,
%\[
%    \left| \log (e^{t^2}\varphi_{c\eta+c'\eta'}(t))\right|\le \left| \log (e^{(ct)^2}\varphi_{\eta}(ct))\right|+
%    \left|\log (e^{(c't)^2}\varphi_{\eta}(c't))\right|;
%\]
%dividing by $|t|^3$, we see that once $cL,c'L\le 1$,
If $cL,c'L\le 1$, taking the logarithm of~\eqref{eq:t-cc} and dividing by $|t|^{\emr{2+\eps}}$, we get
\begin{multline*}
    N_{L\rad}(\eta'') = \sup_{0<|t|< L\rad} \frac{\left|\log \left(e^{\emr{\frac{1}{2}}t^2}\varphi_{\eta''}(t)\right)\right|}{|t|^{\emr{2+\eps}}} \le \\ \le
    \sup_{0<|t|< L\rad} \frac{\left|\log \left(e^{\emr{\frac{1}{2}}(ct)^2}\varphi_{\eta}(ct)\right)\right|}{|t|^{\emr{2+\eps}}} + \sup_{0<|t|< L\rad} \frac{\left|\log \left(e^{\emr{\frac{1}{2}}(c't)^2}\varphi_{\eta'}(c't)\right)\right|}{|t|^{\emr{2+\eps}}}
\end{multline*}
Making the $ct$ and $c't$ variable change in the first and second expressions respectively in the right hand side, we obtain
\begin{multline*}
    N_{L\rad}(\eta'')
    \le  c^{\emr{2+\eps}} N_{cL\rad}(\eta)+ (c')^{\emr{2+\eps}} N_{c'L\rad}(\eta') \le \\
    \le (c^{\emr{2+\eps}}+ (c')^{\emr{2+\eps}}) \cdot \max( N_{\rad}(\eta), N_{\rad}(\eta')) \le
    \\
    \le \max(c^{\emr{\eps}},(c')^{\emr{\eps}})  \cdot \max( N_{\rad}(\eta), N_{\rad}(\eta')).
\end{multline*}

    Taking $L=\sqrt{\frac{C+1}{C}}$ and $\lambda=L^{-\eps}$ concludes the proof.
\end{proof}

%\newpage

\subsection{Correction by an additional term}\label{s:correction}
For the family $\{\txi_{a;n}\}$, one has % {eq:txi-a-n-np}
\[
\txi_{a;n+n'} = (\txi_{a;n} + \txi_{a+n;n'}) - \tR_{a;n,n'},
\]
with an additional term $\tR_{a;n,n'}$ present in addition to the sum of independent random variables
\begin{equation}\label{eq:sum-txis}
\st_{a;n,n'}:=\txi_{a;n}+\txi_{a+n;n'},
\end{equation}
Due to this, an additional (and possibly non-independent) term is added to the normalized random variable: for
\begin{equation}\label{eq:X-Y-S}
    X= \frac{\st_{a;n,n'}}{\sqrt{\var \st_{a;n,n'}}}, \quad Y = \frac{\txi_{a;n+n'}}{\sqrt{\var \txi_{a;n+n'}}}
\end{equation}
this term is the difference
\[
    r=r_{a;n,n'}=Y-X.
\]

%\newpage

In this section we control influence of this difference between normalized variables on the corresponding value $N'_{\rho}(\cdot)$. First,
note that $\E r=0$, and its $2+\eps$-th moment $\E |r|^{\emr{2+\eps}}$ satisfies the following upper bound.
\begin{lemma} \label{lm:errorEst}
Under the assumptions~(\ref{a:center})--(\ref{a:growth}) above, and with $n_0$ given by Lemma~\ref{l:var-lower},
    there exists $Q_r < \infty$, such that for every $n, n'\ge n_0$, satisfying
    \begin{equation*}
        \frac{n}{2} \le n' \le 2n,
    \end{equation*}
%    and any $\mgr_1,\dots,\mgr_{n+n'}\in\mK$,
we have
    \begin{equation}\label{eq:Qr}
        \E |r_{a;n, n'}|^{\emr{2+\eps}} <
        \left(\frac{Q_r}{\sqrt{n + n'}}\right)^{\emr{2+\eps}}
        %Q_r^{\emr{2+\eps}} \cdot (n + n')^{-\emr{(1+\frac{\eps}{2})}}.
    \end{equation}
\end{lemma}

\begin{proof}
    Note that $r_{a;n,n'}$ can be expressed as
\begin{equation}\label{eq.r}
    r=r_{a;n,n'}= \frac{\sqrt{\var \st_{a;n,n'}}-\sqrt{\var \txi_{a;n+n'}}}{\sqrt{\var \txi_{a;n+n'}}} \cdot X - \frac{1}{\sqrt{\var \txi_{a;n+n'}}} \, \tR_{a;n,n'}.
\end{equation}
Now, the $L_{\emr{2+\eps}}$-norms of both $X$ and $\tR_{a;n,n'}$ are uniformly bounded (due to Corollary~\ref{c.3} and assumption~(\ref{a:difference}) respectively). On the other hand, from the Cauchy-Schwartz inequality one has
\[
\left| \sqrt{\var \st_{a;n,n'}}-\sqrt{\var \txi_{a;n+n'}} \right| \le \sqrt{\var \tR_{a;n,n'}},
\]
and hence, due to Lemma~\ref{l:var-lower}, both coefficients admit upper bounds as $\frac{\const}{\sqrt{n+n'}}$.
\end{proof}

Our next step is to control the influence of such a ``small'' change by $r=Y-X$ on the non-Gaussianity value~$N_{\rad}(\cdot)$. We first estimate the change of the characteristic function. Namely, assuming that we are given random variables $X$ and $Y$ with a bound on their $2+\eps$-th moment, we provide an estimate that tends to zero as a power of $2+\eps$-th moment of their difference.

\begin{lemma}\label{l:X-Y}
Let $X,Y$ be two random variables with
\[
\E X = \E Y =0, \quad
\var X = \var Y=1, \quad
\E |X|^{\emr{2+\eps}},\E |Y|^{\emr{2+\eps}} < C_{X}^{\emr{2+\eps}}.
\]
Assume that for $r=Y-X$ one has $\E |r|^{\emr{2+\eps}}<C_{r}^{\emr{2+\eps}}$. Then for any $t\in \R$ one has
\begin{equation}\label{eq:t-bound}
{|\varphi_{X}(t)- \varphi_{Y}(t)|} \le \emr{5} C_r^{\eps} C_X^{2}  \cdot {|t|^{\emr{2+\eps}}}
\end{equation}
\end{lemma}
\begin{proof}
Note that it suffices to obtain an estimate
\begin{equation}\label{eq:t-rho}
|(\varphi_{X }-\varphi_{Y})''(t)| \le \emr{10\, } C_r^{\eps} C_X^{2}  \cdot |t|^{\emr{\eps}};
\end{equation}
indeed, using
\[
\varphi_{X}(t)- \varphi_{Y}(t) = \int_0^t \, dt_1 \int_0^{t_1} (\varphi_{X }-\varphi_{Y})''(t_2) \, dt_2,
\]
and integrating~\eqref{eq:t-rho} two times, we get the desired~\eqref{eq:t-bound}.

Now, the second derivative of the difference of the characteristic functions can be rewritten in the following way.
\begin{equation}\label{eq:ddphi}
(\varphi_X-\varphi_Y)''(t) = -\E (X ^2 e^{itX} - Y^2 e^{i t Y}) = -\E (X ^2 (e^{itX}-1) - Y^2 (e^{i t Y}-1)),
\end{equation}
where the second equality follows from  $\var X=\var Y$.

In order to obtain the estimate~\eqref{eq:t-rho}, let us decompose the right hand side of~\eqref{eq:ddphi}:
\[
|\varphi_{X}''(t)- \varphi_{Y}''(t)| \le  \E |(X ^2-Y^2) (e^{itX}-1)| + \E (Y^2 |e^{i t Y}-e^{i t X}|).
\]
Using~\eqref{eq:exp-ix-bound}, we can estimate the second summand:
\[
\E (Y^2 |e^{i t Y}-e^{i t X}|) \le \E (Y^2 \cdot \emr{2} |t(X-Y)|^{\emr{\eps}}) = |t|^{\emr{\eps}} \cdot \E (\emr{2}|r|^{\emr{\eps}} \cdot Y^2);
\]
the right hand side now does not exceed $\emr{2} C_r^{\eps} C_X^{2} |t|^{\emr{\eps}}$ due to the H\"older inequality with the exponents $p=\frac{2+\eps}{\eps}$ and $q=\frac{2+\eps}{2}$.

To estimate the first summand, note that $Y^2-X^2 = r(X+Y)$, hence it does not exceed
\[
\E |(X ^2-Y^2) (e^{itX}-1)| \le \E \left( |r| \cdot (|X|+|Y|) \cdot \emr{2}|tX|^{\emr{\eps}}\right).
\]
Moreover, as $|r|\le |X|+|Y|$ and hence $|r|\le |r|^{\eps} (|X|+|Y|)^{1-\eps}$, this summand does not exceed
\[
\E \left(|r|^{\eps} \cdot 2 (|X|+|Y|)^2\right) \cdot |t|^{\eps},
\]
and due to the H\"older inequality with the same exponents $p,q$ as above it does not exceed $\emr{8 \,}C_r^{\eps} C_X^2\cdot |t|^{\eps}$.
Adding these estimates together, we obtain the desired~\eqref{eq:t-rho}.
\end{proof}

Now, apply these estimates to estimate the change of $N_{\rho}(\cdot)$:

\begin{prop}\label{p:r-X-Y}
Under the assumptions of Lemma~\ref{l:X-Y}, let $K_r:=\emr{5\,} C_r^{\eps} C_X^{2}$ be the factor that appears in its conclusion~\eqref{eq:t-bound}. Assume additionally that for some $\rad>0$ one has
\begin{equation}\label{eq.r3N}
    \rad^{\emr{2+\eps}} N_{\rad}(X)\le \frac{1}{100}
\end{equation}
and
\begin{equation}\label{eq.Ke}
    K_r \rad^{\emr{2+\eps}} e^{\frac{\rad^2}{2}} \le \frac{1}{100}.
\end{equation}
Then
\[
N_{\rad}(Y) \le N_{\rad}(X) + 2 K_r e^{\frac{\rad^2}{2}}.
\]
\end{prop}

\begin{proof}
    Note first that due to~\eqref{eq.r3N}, for any $|t|\le \rad $ we have
    \[
        \left|\log \left(\varphi_{X}(t) e^{\frac{t^2}{2}}\right)\right| \le \frac{1}{100}
    \]
    and hence
    \[
        \left | \varphi_{X}(t) e^{\frac{t^2}{2}}  -1 \right | \le \frac{1}{50}
    \]
    (as the exponent function is $2$-Lipschitz in $U_{1/100}(0)$).

    At the same time, the conclusion of Lemma~\ref{l:X-Y} and the assumption~\eqref{eq.Ke} imply that for any $|t|\le \rad$ one has
    \[
        \left| \varphi_{X}(t) e^{\frac{t^2}{2}} -
        \varphi_{Y}(t) e^{\frac{t^2}{2}}
        \right|
        =
        \left| \varphi_{X}(t) -
        \varphi_{Y}(t)
        \right| \cdot e^{\frac{t^2}{2}} \le K_r \rad^{\emr{2+\eps}} \cdot e^{\frac{\rad^2}{2}} \le \frac{1}{100},
    \]
    hence altogether
    \[
        \left| \varphi_{Y}(t) e^{\frac{t^2}{2}} - 1\right| \le
        \left| \varphi_{X}(t) e^{\frac{t^2}{2}} -
        \varphi_{Y}(t) e^{\frac{t^2}{2}}
        \right|+
        \left| \varphi_{X}(t) e^{\frac{t^2}{2}} - 1\right| \le \frac{1}{100}+\frac{1}{50} \le \frac{1}{25}.
    \]
    Finally, the logarithm function is $2$-Lipschitz in $U_{\frac{1}{25}}(1)$, and thus for such $t$
    \begin{multline}
        \left|\log \left(\varphi_{Y}(t) e^{\frac{t^2}{2}}\right) \right| \le
        \left|\log \left(\varphi_{X}(t) e^{\frac{t^2}{2}}\right) \right| +
        2 \cdot \left|\varphi_{Y}(t) e^{\frac{t^2}{2}} - \varphi_{X}(t) e^{\frac{t^2}{2}}   \right| \\
        \le N_{\rad}(X) |t|^{\emr{2+\eps}}  + 2K_r e^{\frac{t^2}{2}} |t|^{\emr{2+\eps}}  \le (N_{\rad}(X)+2K_r e^{\frac{t^2}{2}}) \cdot |t|^{\emr{2+\eps}}.
    \end{multline}
    This implies the desired
    \[
        N_{\rad}(Y) \le N_{\rad}(X)+2 K_r e^{\frac{\rad^2}{2}}.
    \]
\end{proof}

Let us now apply Proposition~\ref{p:r-X-Y} to the family $\{\txi_{a;n}\}$
%normalised random variables $\eta_{a;n}$ and $\st_{a;n,n'}$, defined by~\eqref{eq:eta-family} and~\eqref{} respectively,
%$=\frac{1}{\sqrt{\var \txi_{a;n}}}\,\txi_{a;n}$,
%$\frac{\txi_{a;n}}{\sqrt{\var \txi_{a;n}}}$,
that occurs in the setting of Theorem~\ref{t:proba}.

\begin{coro}\label{c:step}
    Under the assumptions of  Theorem~\ref{t:proba},
    there exists a constant $K$ such that if for some $\rad>0$, some $a$, some $n,n'\ge n_0$, \emr{with $n_0$ given by Lemma~\ref{l:var-lower}, as soon as} $\frac{n}{2}\le n'\le 2n$, one has
    %$X=\eta_n$, $Y=\eta_{(n,n+n']}$, where
    \begin{equation}\label{eq.r3N-eta}
        \rad^{\emr{2+\eps}} N'_{\rad}(\txi_{a;n}+\txi_{a+n;n'})\le \frac{1}{100}
    \end{equation}
    and
    %~\eqref{eq.r3N} is satisfied together with
    \begin{equation}\label{eq.K-n-sqrt}
        \rad^{\emr{2+\eps}} \frac{K}{(n+n')^{\emr{\eps/2}}} e^{\frac{\rad^2}{2}} \le \frac{1}{100},
    \end{equation}
    then
    \[
        N'_{\rad}(\txi_{a;n+n'}) \le N'_{\rad}(\txi_{a;n}+\txi_{a+n;n'})+ 2 \frac{K}{(n+n')^{\emr{\eps/2}}} e^{\frac{\rad^2}{2}}.
    \]
\end{coro}
\begin{proof}
	Take $X$ and $Y$, defined by~\eqref{eq:X-Y-S}, so that
	\[
		N_{\rho}(X)=N'_{\rho} (\txi_{a;n}+\txi_{a+n;n'}), \quad 		
		N_{\rho}(Y)=N'_{\rho} (\txi_{a;n+n'}).
	\]
	Applying Lemma~\ref{lm:errorEst}, from its conclusion~\eqref{eq:Qr} we get an estimate for the $L_{2+\eps}$-norm of their difference $r=r_{a;n,n'}=Y-X$:
	\[
		C_r:=\left(\E |r|^{2+\eps}\right)^{1/(2+\eps)} \le \frac{Q_r}{\sqrt{n+n'}}.
	\]
	Hence, the value $K_r=\emr{5} C_r^{\eps} C_X^{2}$ in Proposition~\ref{p:r-X-Y} is bounded from above by
    \[
        K_r = \emr{5\, } C_r^{\eps} C_X^{2} \le \emr{5\, } \left(\frac{Q_r}{\sqrt{n+n'}}\right)^{\emr{\eps}} \cdot C_X^{2} = \frac{K}{ (n+n')^{\emr{\eps/2}}},
    \]
    where
    \[
        K:= \emr{5 \,} Q_r^{\emr{\eps}}C_X^{2}.
    \]
    The conclusion then immediately follows from Proposition~\ref{p:r-X-Y}.
\end{proof}

\section{Proof of the main result}\label{s:main}

Joining the results of the previous sections, we obtain the following proposition: %[Bootstrapping] ??
\begin{prop}\label{p.boot}
Under the assumptions of Theorem~\ref{t:proba}, there exist sequences $\rad_n\to\infty$, $\delta_n \to 0$, such that for any $n\ge n_0$, \emr{where $n_0$ is given by Lemma~\ref{l:var-lower}}, and any $a$, one has
\begin{equation}\label{eq:Np-delta}
    N'_{\rad_n}(\txi_{a;n}) \le \delta_n.
\end{equation}
\end{prop}

Prior to proving it, note that Theorem~\ref{t:proba} follows from Proposition~\ref{p.boot} almost immediately. %Theorem~\ref{t.CLT}

\begin{proof}[Proof of Theorem~\ref{t:proba}]
Assume that the assumptions of Theorem~\ref{t:proba} are satisfied. Due to Proposition~\ref{p.boot}, for any $\rad>0$ for all sufficiently large $n$ we have $\rho_n\ge \rho$, and hence
\[
\lim_{n\to\infty} N'_{\rad}(\txi_{a;n})
=\lim_{n\to\infty} N_{\rad}(\eta_{a;n})
=0,
\]
where
\[
    \eta_{a;n}=\frac{\txi_{a;n}}{\sqrt{\var \txi_{a;n}}}.
\]

In particular, the characteristic functions $\varphi_{\eta_{a;n}}(t)$ of the normalized variables
converge uniformly on compact sets to~$e^{-\frac{t^2}{2}}$. As the weak convergence
of random variables is equivalent (L\'evy's continuity theorem) to the pointwise convergence of their characteristic functions, we have the desired weak convergence
\[
\eta_{a;n}=\frac{\txi_{a;n}}{\sqrt{\var \txi_{a;n}}} \to \mathcal{N}(0,1), \quad n\to \infty.
\]
Moreover, this convergence is actually uniform in~$a$, as the convergence of the characteristic functions is uniform in $a$ on any compact interval $[-\rho,\rho]$ due to the uniform estimate~\eqref{eq:Np-delta}.
\end{proof}

\begin{proof}[Proof of Proposition~\ref{p.boot}]

%We start the sequence by choosing

We are going to construct the sequence $(\rad_n,\delta_n)_{n\ge n_0}$ so that the desired property~\eqref{eq:Np-delta}
can be established by induction on~$n$. Namely, we have the following
\begin{lemma}[Bootstrapping]\label{l.list}
    Let the sequence $(\rad_n,\delta_n)_{n\ge n_0}$, \emr{where $n_0$ is given by Lemma~\ref{l:var-lower}},
    be chosen in such a way that the following conditions hold:
    \begin{itemize}
        \item As $n\to\infty$, one has
        $  \rad_n\to \infty$ and  $\delta_n\to 0$.
        \item For some $n_1\ge 2n_0$, we have
        \[
        \rad_n=\rad'_0, \quad \delta_n= 3 \bC^{\emr{2+\eps}} \quad \text{for all}\quad n=n_0,\dots,n_1,
        \]
        where $\bC$ is given by Corollary~\ref{c.3},
        \[
            \rad'_0=\min \left(\rad_0, \frac{1}{20\bC} \right)
        \]
        and $\rad_0$ is given by Lemma~\ref{lm:Start}.
        \item For every $m>2n_0$, taking $n=\lfloor\frac{m}{2}\rfloor$ and $n'=m-n$,
        one has
        %the following inequalities hold:
            \begin{gather}
                % \label{ineq:1} \max(\delta_{n}, \delta_{n'}) \cdot \rad_{m}^{\emr{2+\eps}} < 1/100;
                % \\
                \label{ineq:2} \rad_{m} \le L \min(\rad_n, \rad_{n'}),
                \\
                \label{ineq:3}
                \rad_m^{\emr{2+\eps}} \delta_m \le \frac{1}{100},
                %\left( \lambda \max(\delta_{n}, \delta_{n'}) + 2\frac{K e^{\frac{\rad_{m}^2}{2}}}{\sqrt{m}} \right) \cdot \rad_{m}^{\emr{2+\eps}} < 1/100;
                \\
                \label{ineq:4} \left( \lambda \max(\delta_{n}, \delta_{n'}) + 2 \,\frac{K e^{\frac{\rad_{m}^2}{2}}}{m^{\emr{\eps/2}}} \right) \le \delta_{m},
            \end{gather}
            where constants $L$ and $\lambda$ are defined by the conclusion of Lemma~\ref{lm:Contr} for $C=2\left(\frac{\CVar}{C_1}\right)^2$, where $C_1$ is given by Lemma~\ref{l:var-lower}, and $\CVar$ is given by~\eqref{eq:C-2}.
            %~\eqref{VarGrowth}.
    \end{itemize}
    Then the conclusion of Proposition~\ref{p.boot} holds for this sequence.
\end{lemma}
\begin{proof}
    The proof of~\eqref{eq:Np-delta} is by induction. Namely, the base of the induction is formed by $m=n_0,\dots,n_1$, where the inequality
    \[
    N'_{\rho_m}(\txi_{a;m})\le \delta_m
    \]
    follows from the choice of $\rad'_0$ and Lemma~\ref{lm:Start}.

    Let us make the induction step. Namely, for $m>n_1$ let $n=\lceil\frac{m}{2}\rceil$ and $n'=m-n$; then, $n_0\le n,n'<m$, so the conclusion is already established for these indices.
%
% , and write
%    \[
%        \st_{a;n,n'}=\txi_{a;n} + \txi_{a+n;n'}.
%    \]
    Hence, due to the induction assumption and inequality~\eqref{ineq:2},
    \[
        N'_{\rad_m/L} (\txi_{a;n}) \le N'_{\rad_n}(\txi_{a;n}) \le \delta_n \quad \text{and} \quad N'_{\rad_m/L} (\txi_{a+n;n'}) \le N'_{\rad_n}(\txi_{a+n;n'}) \le \delta_{n'}.
    \]
    and thus due to the conclusion~\eqref{eq:N-xi-xi-prim} of Lemma~\ref{lm:Contr},
    \begin{equation}\label{eq:N-prim-sum}
        N'_{\rad_m} (\txi_{a;n} + \txi_{a+n;n'})
        %(\st_{a;n,n'})
        \le \lambda \max(\delta_n,\delta_{n'}).
    \end{equation}

    Now, we are going to apply Corollary~\ref{c:step} for $\rad_m,n,n'$.
%    . and
%    \[
%        X=\frac{\theta_{n,n'}-\E\theta_{n,n'}}{\sqrt{\var \theta_{n,n'}}}, \quad Y=\frac{\xi_{n+n'}-\E\xi_{n+n'}}{\sqrt{\var \xi_{n+n'}}}=\eta_m, \quad r=r_{n,n'}=Y-X.
%    \]
    First, let us check that its assumptions are satisfied. Indeed,
%    \[
%	N'_{\rad_m}(\st_{a;n,n'})\le \max(\delta_n,\delta_{n'}) \le \delta_m
%    \]
%    due to~\eqref{ineq:4};
    multiplying~\eqref{eq:N-prim-sum} by $\rad_m^{\emr{2+\eps}}$ and applying~\eqref{ineq:3} and~\eqref{ineq:4}, we get
    \[
        \rad_m^{\emr{2+\eps}} N'_{\rad_m}(\txi_{a;n}+\txi_{a+n;n'}) \le \rad_m^{\emr{2+\eps}} \delta_m \le \frac{1}{100},
    \]
    so the assumption~\eqref{eq.r3N-eta} is satisfied.
    %what proves~\eqref{eq.r3N}.
    Next, assumption~\eqref{eq.K-n-sqrt} follows again from~\eqref{ineq:4} and~\eqref{ineq:3}:
    \[
        \frac{K}{m^{\emr{\eps/2}}} \rad_m^{\emr{2+\eps}}
        e^{\frac{\rad_m^2}{2}} \le \rad_m^{\emr{2+\eps}} \delta_m \le \frac{1}{100}.
    \]
    Corollary~\ref{c:step} is thus applicable, and
    hence (again using~\eqref{ineq:4}) we get
    \[
        N'_{\rad_m}(\txi_{a;m}) \le \left( \lambda \max(\delta_{n}, \delta_{n'}) + 2 \frac{K e^{\frac{\rad_{m}^2}{2}}}{m^{\emr{\eps/2}}} \right) \le \delta_{m}.
    \]
    The induction step is complete.
\end{proof}

%=N_{\rad_m}(\eta_m)
%, as well as an ancillary inequality
%\[
%    \rad_n^3 \delta_n \le \frac{1}{100},
%\]
%
%\newpage

To complete the proof of Proposition~\ref{p.boot}, it remains to construct the sequences
\[
    \rad_m\to \infty, \quad \delta_m \to 0
\]
that satisfy the assumptions of Lemma~\ref{l.list}. Roughly speaking, if we were keeping the radii $\rho_m$ constant, the contraction with the factor~$\lambda$ would effectively allow to bring $\delta_m$ to zero as the additional term $\frac{2K e^{\frac{\rad_m^2}{2}}}{m^{\emr{\eps/2}}}$ then also tends to zero. It suffices now to make the radii $\rad_m$ increase extremely slowly, so that the exponent $e^{\frac{\rad_m^2}{2}}$ would not break this asymptotic vanishing.

Following this idea, we will choose $\rad_m$ so that
\[
    \rad_m \le \frac{1}{2}\sqrt{\emr{\eps} \log m};
\]
such a restriction allows to use
\begin{equation}\label{eq:K-bound}
\frac{K e^{\frac{\rad_m^2}{2}}}{m^{\emr{\eps}/2}} \le \frac{K m^{\emr{\eps}/8}}{m^{\emr{\eps}/2}}< \frac{K}{m^{\emr{\eps}/4}}
\end{equation}
when checking that inequality~\eqref{ineq:4} holds. Next, we let   %~\eqref{ineq:3},
\begin{equation}\label{eq.delta-def}
    \delta_m=A m^{-\beta}, \quad m>n_1,
\end{equation}
where the constant $A$ is chosen so that at $m=n_1$ this value coincides with $3\bC^{\emr{2+\eps}}$,
\begin{equation}\label{eq.A-def}
    A=3\bC^{\emr{2+\eps}} \cdot n_1^{\beta},
\end{equation}
and the (sufficiently small) power $\beta>0$ and the (sufficiently large) initial index $n_1$ are yet to be fixed.

Now, choose the exponent $\beta>0$ sufficiently small so that
\[
\lambda \cdot 2^{\beta}<1, \quad \beta<\frac{\emr{\eps}}{4},
\]
and fix $\lambda'\in (2^{\beta}\lambda,1)$.

Then, for all sufficiently large $n_1$ the condition~\eqref{ineq:4} holds and can be proved by induction. Indeed, in the left hand side the first summand is
\[
\lambda \max(\delta_n,\delta_{n'}) \le \lambda \cdot A \left(\frac{m-1}{2}\right)^{-\beta} = 2^{\beta} \lambda \cdot A m^{-\beta} \cdot \left(\frac{m-1}{m}\right)^{-\beta} < \lambda' \delta_m.
\]
The second summand due to~\eqref{eq:K-bound} is at most $K \cdot m^{-\emr{\eps}/4}$, thus it suffices to check for $m>n_1$ the inequality
\[
\lambda' A m^{-\beta} + 2K m^{-\emr{\eps}/4} <  A m^{-\beta},
\]
which can be rewritten as
\begin{equation}\label{eq.AMK}
(1-\lambda') A m^{-\beta} > 2K m^{-\emr{\eps}/4}.
\end{equation}
%or, equivalently,
%\begin{equation*}
%(1-\lambda') A > K m^{-(\emr{\eps}/4) + \beta}.
%\end{equation*}
As $\beta<\frac{\emr{\eps}}{4}$, due to the monotonicity it suffices to check~\eqref{eq.AMK} for $m=n_1$. Now, recall that~\eqref{eq.A-def} is used to determine $A$ for given~$n_1$ and~$\beta$; substituting the value of $A$ from~\eqref{eq.A-def}, we see that~\eqref{eq.AMK} holds for $m=n_1$ once $n_1$ is sufficiently large to ensure
\[
(1-\lambda') \cdot 3\bC^{\emr{2+\eps}}  > 2K n_1^{-\emr{\eps}/4}
\]

We fix a sufficiently large $n_1$ so that~\eqref{eq.AMK} holds, fix the corresponding $A$ (defined by~\eqref{eq.A-def}) and the sequence $(\delta_m)$, defined for $m>n_1$ by~\eqref{eq.delta-def}. Then, we use~\eqref{ineq:3} and~\eqref{ineq:2} to choose the sequence $(\rad_m)$. Namely, for $m>n_1$ we let
\begin{equation}\label{eq:rad-def}
\rad_m = \min \left(L \min(\rad_n,\rad_{n'}), \frac{1}{2}\sqrt{\emr{\eps} \log m}, (100\delta_m)^{-1/\emr{(2+\eps)}}\right).
\end{equation}
Then the inequality $\rad_m\le (100\delta_m)^{-1/\emr{(2+\eps)}}$ implies~\eqref{ineq:3}, and~\eqref{ineq:2} is satisfied automatically. Finally, as
\[
\min\left(\frac{1}{2}\sqrt{\emr{\eps} \log m}, (100\delta_m)^{-1/\emr{(2+\eps)}}\right) \to \infty, \quad m\to \infty,
\]
the sequence $(\rad_m)$ defined by~\eqref{eq:rad-def} also tends to infinity; actually, it will coincide with $\frac{1}{2}\sqrt{\emr{\eps} \log m}$ for all sufficiently large~$m$.  This completes the proof of Lemma~\ref{l.list}.
%For the constructed sequence $(\rad_m, \delta_m)$ we have checked all the three conditions~\eqref{ineq:2},~\eqref{ineq:3},~\eqref{ineq:4}, thus completing the proof of the lemma.
\end{proof}
%For the sequences $(\rad_m, \delta_m)$, the conditions of Lemma~\ref{l.list}
%%~\eqref{ineq:2}--\eqref{ineq:4}
%are satisfied, and this completes the proof of our main result.

We conclude this section with the proof of our main result, Theorem~\ref{t.CLT}.

\begin{proof}[Proof of Theorem~\ref{t.CLT}]
\emr{Assume that the assumptions of Theorem~\ref{t.CLT} are satisfied. Then, due to Lemma~\ref{l:txi-assumptions}, the family $\txi_{a;n}$, defined by \eqref{eq:txi-def}, satisfies the assumptions~(\ref{a:center})--(\ref{a:difference}), while Proposition~\ref{p:large} guarantees that this family satisfies the assumption~\eqref{a:growth}. Hence, for this family the conclusions of Theorem~\ref{t:proba} hold, implying that random variables
\[
\frac{\xi_n - \E \xi_n}{\sqrt{\var{\xi_n}}} = \frac{\txi_{0;n}}{\sqrt{\var{\txi_{0;n}}}}
\]
weakly converge to $\mN(0,1)$. Moreover, the speed of convergence in Theorem~\ref{t:proba} is regulated by the sequences $(\rho_m,\delta_m)$, constructed in Lemma~\ref{l.list}. These sequences require for their construction only constants $C_R$, $\Cxi$, $n_0$, etc., that can be chosen independently of the sequence of measures $\mgr_1,\mgr_2,\dots\in \mK$, as such uniformity holds for all the statements in Section~\ref{s:R-estimates}. Hence, the convergence of normalised measures in Theorem~\ref{t.CLT} is also uniform in the choice of the sequence $\mgr_1,\mgr_2,\dots\in \mK$.}
\end{proof}

%\newpage

%\vspace{2cm}

\section{CLT for images of vectors and matrix elements}\label{s:CLT-vectors}

This section is devoted to the proof of Theorem~\ref{t:CLT-vectors}.
%we prove the central limit theorem for log-lengths and for the logarithms of absolute values of matrix elements, Theorem~\ref

\begin{proof}[Proof of Theorem~\ref{t:CLT-vectors}]
We start with establishing the first part of the theorem, convergence for normalised log-lengths $\log |T_n v|$. Note that it suffices to show that for every $\pd>0$ there exists a constant $C_{\pd}$ such that for all sufficiently large~$n$,
\begin{equation}\label{eq:P-delta}
\P(\log \|T_n\| - \log |T_n v| >C_{\pd}) <\pd.
%,
%\quad \P(\log \|T_n\| - \log |(T_n)_{i,j}| >C_{\delta}) <\delta.
\end{equation}
Indeed, as $\var (\log \|T_n\|)$ tends to infinity, once~\eqref{eq:P-delta} is established, it would imply that the difference
\[
\frac{\log \|T_n\|-L_n}{\sqrt{\var (\log \|T_n\|)}}-\frac{\log |T_n v|-L_n}{\sqrt{\var(\log\|T_n\|)}} = \frac{\log \|T_n\|-\log |T_n v|}{\sqrt{\var(\log\|T_n\|)}}
\]
converges to zero in probability. And as adding a random variable that converges in probability to zero does not affect the weak limit, this would imply the convergence
\[
\frac{\log |T_n v|-L_n}{\sqrt{\var(\log\|T_n\|)}} \to \mN(0,1).
\]

Now, without loss of generality we can assume that the vector $v$ is a vector of unit length. The difference $\log \|T_n\|-\log |T_n v| = \Theta(T_n,v)$ can then be estimated using Lemma~\ref{l:Theta-bound}: its conclusion~\eqref{eq:Theta-upper} implies that it suffices to show that for a sufficiently small angle $\alpha=\alpha(\pd)$, one has
\begin{equation}\label{eq:r-not-in-U}
\P (\rep(T_n) \in U_{\alpha}{[v]})<\pd.
\end{equation}
Indeed, once such angle $\rho$ is found, we can take $C_{\pd}:=-\log \sin \alpha$.

We will now proceed along the same lines as in the proof of Lemma~\ref{l:B1-B2}. Namely, Lemma~\ref{l:B-preimage} allows us to approximate the location of $\rep(T_n)$: it is $\frac{1}{\|T_n\|^2}$-close to at least one of the directions $[T_n^{-1} e_1]$, $[T_n^{-1} e_2]$.

Now, for every $\alpha>0$ the probability that such a preimage direction belongs to a given $2\alpha$-neighbourhood of a given direction can be estimated using the log-H\"older estimates~\eqref{eq:log-Holder} that are implied by Theorem~\ref{thm:logHol}. Namely, we apply it with
\[
\msp_0^{(1)}=\msp_0^{(2)}=\delta_{[e_i]}, \quad \mgr_j^{(1)}=\mgr_j^{(2)} = (F_2)_* \mgr_{n+1-j},
\]
where $F_2:A\mapsto f^{-1}_A$, so that the measures $\nu_1=\nu_2$ (defined by~\eqref{eq:random-images-msp}) are the distributions of preimages~$[T_n^{-1} e_i]$. Note that
\[
U_{2\alpha}([v]) \times U_{2\alpha}([v]) \subset \{(x,y) \mid \dist(x,y)<4\alpha\},
\]
and thus
\[
\msp_1(U_{2\alpha}([v]))^2 \le (\msp_1\times \msp_2) ( \{(x,y) \mid \dist(x,y)<4\alpha\} ).
\]
Hence, from~\eqref{eq:log-Holder} we get that for every $i=1,2$
\begin{equation}\label{eq:coord-bound}
\P([T_n^{-1} e_i] \in U_{2\alpha}([v])) \le \sqrt{\CM} \cdot  |\log (4\alpha)|^{-\gamma/2}
\end{equation}
Fixing $\alpha>0$ sufficiently small so that the right hand side of~\eqref{eq:coord-bound} does not exceed $\pd/3$, we see that with the probability at least $1-\frac{2}{3}\pd$ neither of the preimages $[T_n^{-1} e_1]$, $[T_n^{-1} e_2]$ belongs to~$U_{2\alpha}([v])$. Finally, take the number $n$ of matrices sufficiently large so that $\|T_n\|^2> \frac{1}{\alpha}$ with the probability at least $1-\frac{\pd}{3}$; by Lemma~\ref{l:B-preimage} it implies that at least one of the distances $\dist(\rep(T_n), [T_n^{-1} e_i])$ doesn't exceed~$\alpha$. Then altogether with the probability at least $1- \pd$ we have
\[
\dist(\rep(T_n), [v]) \ge \max_{i=1,2}(\dist([v], [T_n^{-1} e_i])- \dist(\rep(T_n), [T_n^{-1} e_i])) \ge 2\alpha-\alpha=\alpha,
\]
implying the desired~\eqref{eq:r-not-in-U}.

The same argument applies to the random variables $\log |(T_n)_{i,j}|$:
it suffices to show that for every $\pd>0$ there exists a constant $C'_{\pd}$ such that for all sufficiently large~$n$,
\begin{equation}\label{eq:P}
\P(\log |T_n e_j| - \log |(T_n)_{i,j}| >C'_{\pd}) <\pd.
\end{equation}
Now, for any pair of indices $i,j=1,2$,
\[
|(T_n)_{i,j}| = |T_n e_j| \cdot \sin \dist ( [T_n e_j], [e_{3-i}] ).
\]
Using the same log-H\"older estimate~\eqref{eq:log-Holder} for the forward dynamics, we choose $\alpha'$ such that $[T_n e_j]$ belongs to $\alpha'$-neighbourhoods of either of $[e_1],[e_2]$  with the probability at most~$\pd$, providing~\eqref{eq:P} and thus completing the proof.
\end{proof}

%\newpage

\section{Unboundedness of Variance: proof of Proposition \ref{p:large}} \label{section:variance}

In this section we prove Proposition \ref{p:large}, i.e. show that under the assumptions of Theorem \ref{t.CLT} variances of $\xi_n$  become arbitrarily large.

% \begin{prop}\label{p:large}
%     Under the assumptions of Theorem \ref{t.CLT}, for any $c>0$ there exists $n_1\in \mathbb{N}$ such that for any $n\ge n_1$ and any collection of distributions $\mgr_1,\dots,\mgr_n\in \mathcal{K}$ one has
%     \[
%     \var_{\mgr_1,\dots,\mgr_n} \xi_{n} \ge c.
%     \]
% \end{prop}

In order to do so, we will assume that $n$ is quite large, and will decompose the full product $A_n\dots A_1$ into a several ``long'' groups $D_{m+1},\dots, D_1$, between which some ``short'' compositions are applied:
\[
A_n\dots A_1 = D_{m+1}(B_{\mu_{n_0,m}}\ldots B_{\mu_{1,m}})D_m\ldots D_2(B_{\mu_{n_0,1}}\ldots B_{\mu_{1,1}})D_1.
\]
We will show that (for an appropriate choice of lengths) even conditionally to all $D_1,\dots,D_{m+1}$, the distribution of the log-norm of the product (with high probability) has sufficiently high variance. At the same time, dividing by the product of norms of $D_j$, we get the composition
\[
\frac{D_{m+1}}{\|D_{m+1}\|}(B_{\mu_{n_0,m}}\ldots B_{\mu_{1,m}})\frac{D_m}{\|D_m\|}\ldots \frac{D_2}{\|D_2\|}(B_{\mu_{n_0,1}}\ldots B_{\mu_{1,1}})\frac{D_1}{\|D_1\|},
\]
where all the quotients $\frac{D_j}{\|D_j\|}$ are almost rank-one matrices.

Therefore, we first consider the variance of a distribution of images of a given vector under random linear maps of rank one. In this case it is easier to show that the variance grows, see Lemma \ref{l:eps0}  and Lemma  \ref{l:m-eps0}  below. By continuity, if one replaces random rank one linear maps by random linear maps of large norm, and uses the fact that for a matrix $D\in \SL(2, \mathbb{R})$ with large norm, $\frac{D}{\|D\|}$ is close to a linear map of rank one and norm one, then a lower bound on variances still holds, see Lemma \ref{l:U} and Corollary \ref{c:Q}. Finally, we can complete the proof of Proposition \ref{p:large} by applying the fact that with large probability a composition of a long enough sequence of random $\SL(2, \mathbb{R})$ matrices has a large norm.

Let us now realize this strategy.

Let $Y\subseteq GL_2(\mathbb{R})$ be the space of all linear maps $\mathbb{R}^2\to \mathbb{R}^2$ of norm~$1$ and of rank~$1$. Notice that $Y$ is homeomorphic to the torus $\mathbb{T}^2$; indeed, it follows from the fact that any such map can be represented as a composition of an orthogonal projection to a one-dimensional subspace and a rotation.

\begin{lemma}\label{l:eps0}
There exist $\eps_0>0$ and $n_0\in \mathbb{N}$ such that for any non-zero vector $v\in \mathbb{R}^2$, any $p\in Y$, and any $\mu_1, \mu_2, \ldots, \mu_{n_0}\in \mathcal{K}$ we have
$$
\var \log\left|p\circ(B_{\mu_{n_0}}\ldots B_{\mu_1})v\right|\ge \eps_0.
$$
\end{lemma}
%\begin{proof}
%TO DO
%\end{proof}

%\begin{lemma}
%The effect of each $B$ is bounded away from zero.
%\end{lemma}

To prove Lemma \ref{l:eps0} we will use a statement from [GK] that was called {\it Atom Dissolving Theorem} there. We will start with a couple of definitions.

\begin{defi}\label{d.defMax}
Denote by $\mathfrak{Max}(\msp)$ the weight of a maximal atom of a probability measure $\msp$. In particular, if $\msp$ has no atoms, then $\mathfrak{Max}(\msp)=0$.
\end{defi}
\begin{defi}
Let $X$ be a metric compact. For a measure $\mgr$ on the space of homeomorphisms $\Homeo(X)$, we say that there is
\begin{itemize}
\item \emph{no finite set with a deterministic image}, if there are no two finite sets $F,F'\subset X$ such that $f(F)=F'$ for $\mgr$-a.e. $f\in \Homeo(X)$;
\item \emph{no measure with a deterministic image}, if there are no two probability measures $\msp,\msp'$ on $X$ such that $f_*\msp=\msp'$ for $\mgr$-a.e. $f\in \Homeo(X)$.
\end{itemize}
\end{defi}

The following statement is a general statement for non-stationary dynamics, ensuring the ``dissolving of atoms'': decrease of the  probability of a given point being sent to any particular point.

\begin{theorem}[Atoms Dissolving Theorem~\emr{1.13} from {\cite{GK}}]\label{p.max.full}
Let $\bK_X$ be a compact set of probability measures on $\Homeo(X)$.
\begin{itemize}
\item
Assume that for any $\mgr\in \bK_X$ there is no finite set with a deterministic image. Then for any $\eps>0$ there exists $n$ such that for any probability measure $\msp$ on $X$ and any sequence $\mgr_1, \ldots, \mgr_n\in \bK_X$ we have
$$
%\mathfrak{Max}\left(\mathbb{E}_{\mgr_{n}, \ldots, \mgr_{1}}(F_n)_*\msp\right) <\eps.
\mathfrak{Max}\left(\mgr_{n}* \dots* \mgr_{1}*\msp\right) <\eps.
$$
In particular, for any probability measure $\msp$ on $X$ and any sequence $\mgr_1, \mgr_2, \ldots\in \bK_X$ we have
$$
\lim_{n\to \infty}\mathfrak{Max}\left(\mgr_{n}* \dots* \mgr_{1}*\msp\right)=0.
$$
\item If, moreover, for any $\mgr\in \bK_X$ there is no measure with a deterministic image, then the convergence is exponential and uniform over all sequences $\mgr_1, \mgr_2, \ldots$ from $K^{\mathbb{N}}$ and all probability measures $\msp$. That is, there exists $\lambda<1$ such that for any $n$, any $\msp$ and any $\mgr_1,\mgr_2,\dots\in\bK_X$
\[
\mathfrak{Max}\left(\mgr_{n}* \dots* \mgr_{1}*\msp\right) < \lambda^n.
\]
\end{itemize}
\end{theorem}

In the proof below we will only be using the first part of Theorem \ref{p.max.full}.

\begin{proof}[Proof of Lemma~\ref{l:eps0}]
Due to Theorem \ref{p.max.full} and our assumptions regarding the measures from $\mathcal{K}$, there exists $n'\in \mathbb{N}$ such that for any $\mu_1, \mu_2, \dots, \mu_{n'}\in \mathcal{K}$ we have
$$
\mathfrak{Max}\left(\mgr_{n}* \dots* \mgr_{1}*\msp\right) <\frac{1}{2}
$$
for any probability measure $\msp$ on $\mathbb{RP}^1$. To prove Lemma \ref{l:eps0} it is enough to choose $n_0=n'+1$.

Since
 $$\var \log\left|p\circ(B_{\mu_n}\ldots B_{\mu_1})v\right|=\var \log\left|p\circ(B_{\mu_{n}}\ldots B_{\mu_1})\frac{v}{|v|}\right|,$$
without loss of generality we can assume that $|v|=1$ and, slightly abusing the notation, consider it  an element of $\mathbb{RP}^1$. For given $p\in Y$, $v\in \mathbb{R}^2, |v|=1$, $ \{\mu_1, \mu_2, \ldots, \mu_{n_0}\}\in \mathcal{K}^{n_0}$ consider the probability distribution $\chi$ on $[0, +\infty)$ of the random images $\left|p\circ(B_{\mu_{n_0}}\ldots B_{\mu_1})v\right|$.
\begin{lemma}\label{l.semi}
The function $\Phi: \mathbb{RP}^1\times Y\times \mathcal{K}^{n_0}\to \mathbb{R}\cup \{\infty\}$ defined by
  $$
\Phi(v, p, \mu_1, \mu_2, \ldots, \mu_{n_0})=\left\{
  \begin{array}{ll}
    \infty, & \hbox{if $\chi(\{0\})>0$;} \\
    \var \log\left|p\circ(B_{\mu_{n_0}}\ldots B_{\mu_1})v\right|, & \hbox{if $\chi(\{0\})=0$,}
  \end{array}
\right.
$$
is lower semicontinuous.
\end{lemma}
\begin{proof}
  Notice that $\chi$ depends continuously on $(v, p, \mu_1, \mu_2, \ldots, \mu_{n_0})$ in weak-$*$ topology.

  Let us consider the cases when $\chi(\{0\})>0$ and when $\chi(\{0\})=0$ separately.

Assume first that $\chi(\{0\})>0$. We want to show that given $M>0$, for any sufficiently small perturbation $\chi'$ of $\chi$ we have $ \var \log \chi'>M$. Notice that the measures condition implies that $\chi$ cannot be concentrated exclusively at $0\in \mathbb{R}$. Hence for some $\tau>0$ we have $\chi[\tau, +\infty)>0$. If $\chi'$ is a probability distribution that is sufficiently close to $\chi$, then $\chi'[\tau/2, +\infty)$ is not less than $\frac{1}{2}\chi[\tau, +\infty)$, and the $\chi'$-weight of a small neighborhood of the origin is at least $\frac{1}{2}\chi(\{0\})$. Choosing that neighborhood small enough guarantees that $\var \log \chi'>M.$

Assume now that $\chi(\{0\})=0$ and $\var\log\chi<\infty$. Then
$$
\var \log\chi=\lim_{T\to \infty}\var \left[(\log \chi)|_{[-T, T]}\right],
$$
so for any $\eps>0$, for some large enough $T>0$ we have
$$
\var \left[(\log \chi)|_{[-T, T]}\right]>\var \log\chi -\frac{\eps}{2}.
$$
Therefore, for any $\chi'$ that is sufficiently close to $\chi$ we have
$$
\var \log\chi'\ge \var \left[(\log \chi')|_{[-2T, 2T]}\right]\ge \var \left[(\log \chi)|_{[-T, T]}\right]-\frac{\eps}{2}>\var\log\chi - {\eps}.
$$
The case when $\chi(\{0\})=0$ and $\var\log\chi=\infty$ can be treated similarly.
\end{proof}
% Consider a function
% $$
% \Phi: \mathbb{RP}^1\times Y\times \mathcal{K}^{n_0}\to \mathbb{R}
% $$
% defined by
% $$
% \Phi(v, p, \mu_1, \mu_2, \ldots, \mu_{n_0})=\var \log\left|p\circ(B_{\mu_n}\ldots B_{\mu_1})v\right|.
% $$
The space $\mathbb{RP}^1\times Y\times \mathcal{K}^{n_0}$ is compact. Hence, Lemma \ref{l.semi} implies that it is enough to show that $\Phi>0$ to ensure that for some $\eps_0>0$ we have $\Phi\ge \eps_0>0$.

Suppose this is not the case, and for some unit vector $v$, a linear map $p\in Y$, and $\mu_1, \ldots, \mu_{n_0}\in \mathcal{K}$ we have
$$
\var \log\left|p\circ(B_{\mu_{n_0}}\ldots B_{\mu_1})v\right|=0.
$$
Then for some $d\ge 0$  with probability 1 we have $\left|p\circ(B_{\mu_{n_0}}\ldots B_{\mu_1})v\right|=d.$ That means that $B_{\mu_{n_0}}\ldots B_{\mu_1}v$ has to belong to $L=\{u\ |\ |p(u)|=d\}$, which is a line (if $d=0$) or a union of two lines (if $d>0$). This implies that $\mu_1\times \mu_2\times \ldots \times \mu_{n_0-1}$-almost surely the image $B_{\mu_{n_0-1}}\ldots B_{\mu_1}v$ must belong to the set $\cap_{A\in \supp(\mu_{n_0})}A^{-1}(L)$.
\begin{figure}[h]
\includegraphics{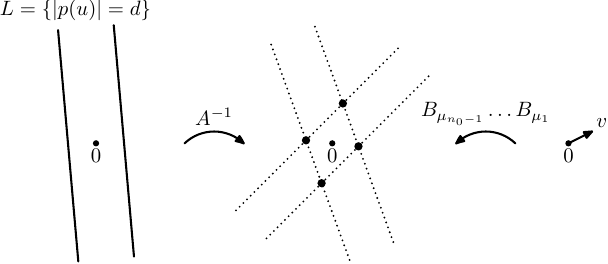}%[width=0.7\textwidth]
\caption{The set $L$ and its preimages}\label{fig:projections}
\end{figure}
Since the measure $\mu_{n_0}$ must satisfy the measure condition, this intersection must consist of at most of four points (see Fig.~\ref{fig:projections}), whose projectivization gives at most two points on $\mathbb{RP}^1$. But this would imply that if $\msp$ is an atomic measure on $\mathbb{RP}^1$ at the point corresponding to the initial vector $v$, then $\mgr_{n_0-1}* \dots* \mgr_{1}*\msp$ is a measure supported on at most two points, which contradicts  the choice of $n_0$ above. This completes the proof of Lemma \ref{l:eps0}.
%As length $l_A$ gets large, $A$'s become similar to rank~$1$ matrices. The effect then tends to the variance of $\eta:=\log |u(Bv)|$, where $v$ is a nonzero vector (associated to~$A_i$) and $u$ a nonzero covector  (associated to~$A_{i+1}$). The variance of $\eta$ might be zero only if $v$ is deterministically sent to a union of two lines. Splitting $B$ into two parts $B=B_{(2)}B_{(1)}$, we ensure that the second part doesn't send this union to another union of two lines (it suffices to ensure that the projectivization of $B_{(2)}^{-1}$ doesn't send any direction to a deterministic one), then the first part $B_{(1)}$ with positive probability doesn't send $v$ to one of the common points of these preimages.
%
%Also, non-zero $v$ can be assumed to be a unit vector, as its normalization shifts the logarithm by a constant and hence does not affect the variance. Finally, the variance is lower semicontinuous in terms of convergence in law, and the set
%\[
%\{|v|=1\} \times Y \times \mK^{n_0}
%\]
%is a compact; hence, as the variance is strictly positive at every point of this compact, it is bounded away from zero by some constant $\eps_0$.
\end{proof}

\begin{lemma}\label{l:m-eps0}
For any $m\in \mathbb{N}$, any $\{\mu_{1, i}, \ldots, \mu_{n_0, i}\}_{i=1, \ldots, m}\in \mathcal{K}^{mn_0}$, and any $\{p_1, \ldots, p_{m+1}\}\in Y^{m+1}$ we have
$$
\var\log\|p_{m+1}(B_{\mu_{n_0,m}}\ldots B_{\mu_{1,m}})p_m\ldots p_2(B_{\mu_{n_0,1}}\ldots B_{\mu_{1,1}})p_1\|\ge \eps_0m.
$$
\end{lemma}
\begin{proof}
As each $p_j$ is a unit norm rank~$1$ matrix, it can be written as
\[
p_j=v_j \otimes \ell_j, \quad \text{where} \quad v_j\in\R^2, \quad \ell_j\in (\R^2)^*, \quad |v_j|=|\ell_j|=1.
\]
Now, let
\[
\tB_j:=B_{\mu_{n_0,j}}\ldots B_{\mu_{1,j}}
\]
be the $j$-th intermediate product. Then for the product
\[
T=p_{m+1}\tB_m p_m\ldots p_2 \tB_1 p_1
\]
one has for any $v\in \R^2$
\[
T(v) = v_{m+1} \cdot \ell_{m+1}(\tB_m v_m) \cdot \dots \cdot \ell_{2}(\tB_1 v_1) \cdot \ell_1(v),
\]
and hence
\begin{equation}\label{eq:T-norm}
\log \|T\| = \sum_{j=1}^m \log |\ell_{j+1}(\tB_j v_j) |.
\end{equation}
Right hand side of~\eqref{eq:T-norm} is a sum of $m$ independent random variables, and the variance of each of them is at least $\eps_0$ due to Lemma~\ref{l:eps0}. Thus, the variance of $\log\|T\|$ is at least~$m\eps_0$.
\end{proof}

\begin{lemma}\label{l:U}
There exists a neighborhood $U$ of the compact $\mK^{n_0m}\times Y^{m+1}$ in $\mK^{n_0m}\times Mat_2(\mathbb{R})^{m+1}$ such that for any
$$
\bar\mu\times \{D_j\}_{j=1, \ldots, m+1}\in U,
$$
where $\bar\mu=\{\mu_{1, i}, \ldots, \mu_{n_0, i}\}_{i=1, \ldots, m}$ and $D_j\in Mat_2(\mathbb{R})$, we have
$$
\var\log\|D_{m+1}(B_{\mu_{n_0,m}}\ldots B_{\mu_{1,m}})D_m\ldots D_2(B_{\mu_{n_0,1}}\ldots B_{\mu_{1,1}})D_1\|\ge \frac{\eps_0m}{2}
$$
\end{lemma}
\begin{proof}
On $\mK^{n_0m}\times Y^{m+1}$ this variance is bounded from below by~$\eps_0 m$ due to Lemma~\ref{l:m-eps0}. As the set $\mK^{n_0m}\times Y^{m+1}$ is compact, and the variance is a lower-semicontinuous function of a distribution, there exists a neighbourhood~$U$ of this compact on which the variance is at least~$\frac{m\eps_0}{2}$.
\end{proof}
\begin{coro}\label{c:Q}
There exists $Q$ such that for any $D_1,\dots, D_{m+1}\in \SL(2,\R)$ with $\|D_j\|\ge Q, \, j=1,\dots, m+1$ one has
$$
\var\log\|D_{m+1}(B_{\mu_{n_0,m}}\ldots B_{\mu_{1,m}})D_m\ldots D_2(B_{\mu_{n_0,1}}\ldots B_{\mu_{1,1}})D_1\|\ge \frac{\eps_0m}{2}
$$
\end{coro}
\begin{proof}
\begin{multline}
\log\|D_{m+1}(B_{\mu_{n_0,m}}\ldots B_{\mu_{1,m}})D_m\ldots D_2(B_{\mu_{n_0,1}}\ldots B_{\mu_{1,1}})D_1\| = \\
=\log\| \tD_{m+1}(B_{\mu_{n_0,m}}\ldots B_{\mu_{1,m}})\tD_m\ldots \tD_2(B_{\mu_{n_0,1}}\ldots B_{\mu_{1,1}})\tD_1\|+ \sum_{j=1}^m \log \|D_j\|,
\end{multline}
where $\tD_j:=\frac{D_j}{\|D_j\|}$. On the other hand, as $\|D\|\to\infty$ for $D\in \SL(2,\R)$, one has $\frac{D}{\|D\|} \to Y$, so it suffices to choose $Q$ sufficiently large to ensure that
\[
(\{\mgr_{i,k}\}_{1\le i\le n_0, 1\le k\le m},(\tD_1,\dots,\tD_{m+1}) )\in U
\]
once $\|D_j\|\ge Q$ for all $j=1,\ldots, m+1$, where $U$ is provided by Lemma~\ref{l:U}.
\end{proof}

\begin{proof}[Proof of Proposition~\ref{p:large}]
First, fix $n_0$ and $\eps_0$ given by Lemma~\ref{l:eps0}. Then, choose and fix $m$ such that $\frac{m\eps_0}{4}>c$.

Now, take a sufficiently large $Q$ provided by Corollary~\ref{c:Q}. % so that if $D_1,\dots, D_{m+1}\in \SL(2,\R)$
It follows from \cite[{Theorem~2.2}]{G} that for a sufficiently large $n_2$ one has
\[
\forall n'\ge n_2 \quad \forall \mgr_1,\dots,\mgr_{n'} \in \mK\quad \Prob_{\mgr_1,\dots,\mgr_{n'}} (\|A_{n'}\dots A_1\| \ge Q) \ge 1- \frac{1}{2(m+1)}.
\]
Now, take $n_3:=n_2(m+1)+n_0 m$. Then, for any $n\ge n_3$ and any $\mgr_1,\dots,\mgr_n\in\mK$ we can decompose the product $A_n\dots A_1$ as
\[
A_n\dots A_1 = D_{m+1} \tB_m D_m \dots D_2 \tB_1 D_1,
\]
where each $D_j$ is a product of at least $n_2$ matrices $A_i$, and each $\tB_j$ is a product of~$n_0$ of $A_i$'s.

This implies that with the probability at least $\frac{1}{2}$ one has $\|D_j\|\ge Q$ for all $j$, and hence the variance of the distribution conditional to such $D_j$ is at least $\frac{m\eps_0}{2}$. Thus, we finally have
\[
\var \xi_n \ge \E_{D_1,\dots D_{m+1}} \var (\xi_{n} \mid D_1,\dots D_{m+1}) \ge \frac{1}{2} \cdot \frac{m\eps_0}{2} = \frac{m\eps_0}{4}>c.
\]
\end{proof}

%\newpage

\section{Examples}\label{s:examples}

In this section, we discuss examples, showing that in the non-stationary setting, assuming only second moments to be uniformly bounded does not suffice to obtain the convergence to the Gaussian distribution. In particular, it shows that the moment assumption in Theorem~\ref{t.CLT} is optimal.

We start with an example addressing the classical CLT setting on sums of independent random variables. For every $\eps>0$, consider a random variable $\xi_{\eps}$, taking values
\begin{equation}\label{eq:eps-xi}
\xi_{\eps} = \begin{cases}
1 & \text{with probability $\frac{1-\eps}{2}$}
\\
-1 & \text{with probability $\frac{1-\eps}{2}$}
\\
\frac{10}{\sqrt{\eps}} & \text{with probability $\eps$}.
\end{cases}
\end{equation}
Now, the second moments of these random variables are uniformly bounded:
\[
\E \xi_{\eps}^2 < 1+ \eps\cdot \frac{10^2}{\eps} = 101.
\]

\begin{example}\label{ex:no-CLT-1D}
Take a sequence $n_j$, satisfying $n_{j+1}>n_j^6$ for all $j$, and let $\eps_j:=\frac{1}{n_j}$. Consider the sequence of independent random variables
\[
\xi_{1,1},\dots, \xi_{1,n_1};\xi_{2,1},\dots, \xi_{2,n_2}; \dots ; \xi_{j,1} \dots, \xi_{j,n_j};\dots
\]
where all the random variables $\xi_{j,i}$, $i=1,\dots,n_j$ have the distribution~\eqref{eq:eps-xi} with $\eps=\eps_j$. Then, though this is a sequence of independent random variables with uniformly bounded second moments, their sums are not asymptotically normal.
\end{example}

\begin{proof}
Consider the sums within the groups,
\[
S_j:=\xi_{j,1} +\dots + \xi_{j,n_j}.
\]
Note that $\frac{1}{\sqrt{n_j}}S_j$ converge in law to a law that is the sum of the Gaussian law $\mathcal{N}(0,1)$ and of the 10-scaled Poisson distribution with the parameter~$1$. Indeed, the Poisson component comes from the variables taking values $\frac{10}{\sqrt{\eps_j}}$, and conditionally to the places where these ``large'' values occur, the part that is left satisfies the classical~CLT.

Thus, the limit law is multi-modal (see Fig.~\ref{fig:multi}), and hence the sums $S_j$ are not asymptotically normal.

\begin{figure}
\includegraphics[width=7cm]{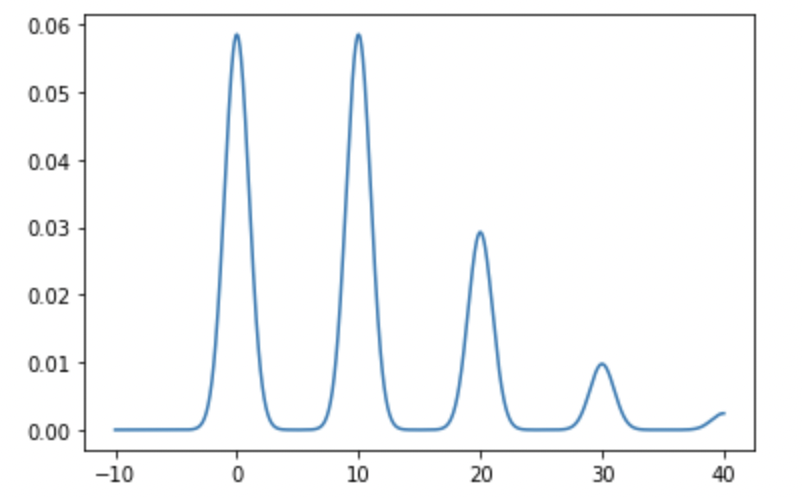}
\caption{Density for the limit law in Example~\ref{ex:no-CLT-1D}}\label{fig:multi}
\end{figure}

Finally, the growth condition $n_j>n_{j-1}^6$ shows that $\frac{S_1+\dots+S_{j-1}}{\sqrt{n_j}}$ converges to zero uniformly, hence the same non-normal distribution is the limit of the sequence
\[
\frac{\xi_{1,1}+\dots+\xi_{j,n_j}}{\sqrt{n_j}} = \frac{S_1+\dots+S_j}{\sqrt{n_j}}.
\]
\end{proof}

Now, for every $\eps\in [0, \frac{1}{2}]$ let $\mgr_{\eps}$ be a measure on $\SL(2,\R)$ that is the law of the random matrix
\[
B_{\eps}= R_{\beta_1} \left(\begin{matrix} r_{\eps} & 0 \\ 0 \,  & r_{\eps}^{-1} \end{matrix}\right) R_{\beta_2},
\]
where $\beta_1$, $\beta_2$ and $r_{\eps}$ are independent, $\beta_1$ and $\beta_2$ are uniformly distributed on $[0,2\pi]$, and
\begin{equation}\label{eq:r-def}
r_{\eps}=\begin{cases}
2 & \text{with probability } \, 1-\eps
\\
e^{\frac{10}{\sqrt{\eps}}} & \text{with probability } \, \eps.
\end{cases}
\end{equation}
These measures then form a compact set $\mK$ of probability measures on~$\SL(2,\R)$, on which the second logarithmic moment is uniformly bounded:
\[
\E (\log \|B_{\eps}\|)^2 < (\log 2)^2 + \eps\cdot \frac{10^2}{\eps} < 101.
\]
Moreover, due to the presence of random rotations by $\alpha$ and $\beta$ the measures condition (stated in Sec.~\ref{s:prelim}) is also satisfied.

Now, denote by $\lambda$ the Lyapunov exponent for the product of the matrices corresponding to~$\mu_0$, and by~$\sigma^2$ the corresponding variance in the CLT. We then have the following example.
\begin{example}\label{ex:no-CLT-2D}
Take lengths $n_j$ and the associated values $\eps_j$ as in Example~\ref{ex:no-CLT-1D}, and take the sequence in which these measures are repeated by groups,
\[
\underbrace{\mu_{\eps_1},\dots, \mu_{\eps_1}}_{n_1 \, \text{times}}; \underbrace{\mu_{\eps_2},\dots, \mu_{\eps_2}}_{n_2 \, \text{times}}; \dots ; \underbrace{\mu_{\eps_j},\dots, \mu_{\eps_j}}_{n_j \, \text{times}}; \dots.
\]
Let $T_n$ be the associated products of independent random matrices, and denote $N_j:=n_1+\dots+n_j$ the index at which the $j$-th group ends. Then the random variables
\[
\frac{1}{\sqrt{n_j}} (\log \|T_{N_j}\| - n_j \lambda)
\]
converge in law to the sum of a Gaussian distribution $\mathcal{N}(0,\sigma^2)$ and of the 10-scaled Poisson distribution with parameter~$1$. In particular, this limit is non-normal, and thus Theorem~\ref{t.CLT} cannot hold for such a product.
\end{example}

\begin{proof}
Let us first instead of the log-norms of matrices $\log \|T_n\|$ fix an initial vector $v_0$ and consider log-lengths $\log |T_n v_0|$ of its images $v_n:=T_n v_0$. Note that increments $\eta_n:=\log |v_{n}|-\log |v_{n-1}|$ of this sequence are then \emph{independent} random variables. Indeed, due to the pre-composition with the rotation $R_{\beta_2}$, the distribution of the increment of the log-norm,
\begin{equation}\label{eq:incr}
\log |B_{\eps} v| - \log |v| = \log \frac{|B_{\eps}v|}{|v|},
\end{equation}
does not depend on the choice of the initial nonzero vector~$v$. Hence, each new ($n$-th) increment $\log \frac{|v_{n+1}|}{|v_n|}$, conditioned to all the previous matrices (that determine the corresponding image~$v_n$ and all the preceding increments) has the same distribution, and thus these increments are independent.

Now, the distribution of the increment~\eqref{eq:incr} is very close to the one in Example~\ref{ex:no-CLT-1D}: it is a mix of a given non-degenerate bounded law (corresponding to the distribution~$\mu_0$) with the probability $1-\eps$ (the case if $r_\eps$ takes value~$2$), and of the law that is associated to the application of the diagonal matrix associated to~$e^{\frac{10}{\sqrt{\eps}}}$. The latter law is concentrated around~$\frac{10}{\sqrt{\eps}}$: as $\eps\to 0$, this law, rescaled by $\sqrt{\eps}$, converges to the constant~$10$. Meanwhile, to the former law the usual Central Limit Theorem is applicable, and the sum of $n_j$ such independent random variables, from which $n_j\lambda$ is subtracted, after division by~$\sqrt{n_j}$ converges to $\mathcal{N}(0,\sigma^2)$.

Also, note that in the same way as in Example~\ref{ex:no-CLT-1D}, it suffices to study the product only of matrices corresponding to the $j$-th group. Indeed, the log-norm of the product of the previous ones is does not exceed $2n_{j-1}\cdot 10 \sqrt{n_{j-1}}=20n_{j-1}^{3/2}$, while $n_{j-1}<n_j^{1/6}$, and thus this upper bound is $o(\sqrt{n_j})$.
%bounded by a constant that is~$o(\sqrt{n_j})$: the log-norms in this product are uniformly bounded by $10\sqrt{n_{j-1}}$, and the number of matrices multiplied doesn't exceed~$2n_{j-1}$, while

Hence, we have the convergence in law of the rescaled log-laws
\begin{equation}\label{eq:T-law}
\frac{1}{\sqrt{n_j}} (\log |T_{N_j} v_0| - n_j \lambda)
\end{equation}
to the sum of $\mathcal{N}(0,\sigma^2)$ and of the 10-scaled Poisson distribution with the parameter~$1$; we denote this distribution~$\mathcal{D}$.

Finally, take the initial vector $v_0$ to be chosen randomly uniformly on the unit circle (and independently from the product~$T_{N_j}$). Conditionally to each choice of $v_0$, the law of~\eqref{eq:T-law} is the same and converges to $\mathcal{D}$ as $j\to \infty$. At the same time, the log-norm $\log \|T_n\|$ is close to $\log |T_n v_0|$ for most (in the sense of the Lebesgue measure) vectors~$v_0$ on the circle. Namely, due to Lemma~\ref{l:Theta-bound},
\begin{equation}\label{eq:T-n-v-0}
\left| \log |T_n v_0| - \log \|T_n\|\, \right| \le |\log \sin \dist ([v_0],\rep(T_n)) |,
\end{equation}
and as $v_0$ is independent from $T_n$ and the distribution of its direction~$[v_0]$ is uniform, the distribution of the random variable in the right hand side of~\eqref{eq:T-n-v-0} does not depend on~$n$. In particular, for every $R$ the probability that the right hand side exceeds $R$ is at most~$e^{-R}$.

%$\delta>0$ there is a constant $R_{\delta}$ such that it does not exceed $R_{\delta}$ with the probability at least $1-\delta$ (actually, one can take $R_{\delta}=\log 2\delta^{-1}$).

%writing $B=T_n$ in the form~\eqref{eq:B-diag}, we see that it suffices to check such a statement for a diagonal matrix~$B$; then, for $v_0=\left( {\cos \theta \atop \sin \theta} \right)$, one gets $|Bv_0| \ge \|B\|\cdot |\cos \theta|$, thus for any given $R>0$ except for the arcs $|\cos \theta| < e^{-R}$, one has $\log |Bv_0|\ge \log \|B\| - R$.

Hence, the convergence in law of the log-lengths~\eqref{eq:T-law} implies also the convergence of
\begin{equation}\label{eq:T-norm}
\frac{1}{\sqrt{n_j}} (\log \|T_{N_j}\| - n_j \lambda)
\end{equation}
to the same distribution~$\mathcal{D}$. Indeed, for $v_0$ chosen independently from $T_{N_j}$, taking $R_j=\sqrt[4]{n_j}$, we see that the random variables~\eqref{eq:T-law} and~\eqref{eq:T-norm} differ at most by $\frac{R_j}{\sqrt{n_j}}=n_j^{-1/4}=o(1)$ on the set of the probability $1-e^{-R_j}=1-o(1)$. In particular, this convergence implies that the sequence of random variables~$\log \|T_n\|$ is not asymptotically normal.
%
%Now, the number of factors for which the corresponding $r_{\eps_j}$ takes value $e^{\frac{10}{\sqrt{\eps_j}}}$ tends to the Poisson distribution, and except for the arbitrarily small probability event, between such factors we apply long products of $\mu_0$-distributed matrices. Almost-additivity of the log-norm of the product then reduces this example to the same setting as in the previous one: the distribution of
%\[
%\log \|A_{j,1}\dots A_{j,n_j}\|
%\]
%is approximated by the sum of the Poisson distribution with scaling $\frac{10}{\sqrt{\eps_j}}$ and of $\mathcal{N}(n_j \lambda, n_j \sigma^2)$, and after subtracting $n_j \lambda$ and dividing by $\sqrt{n_j}$ we obtain the desired convergence is law.
%
%One can also remark that for the log-growth of the lengths of images of vectors, the reduction to the previous example is even more straightforward, as due to the presence of random rotations the length $\log |T_n v|$ for any $v$ is a sum of independent random variables, responsible to the log-length change on every step.
\end{proof}

\end{document}